\documentclass[12pt]{amsart}
\usepackage{amsmath}
\usepackage{amsfonts}
\usepackage{amssymb}
\usepackage{graphicx,caption}

\newcommand{\ip}[2]{\ensuremath{\left<#1,#2\right>}}
\newcommand{\sett}[1]{\ensuremath{\left \{ #1 \right \}}}
\newcommand{\abs}[1]{\ensuremath{\left| #1 \right| }}

\newcommand{\LtRt}{L^2(\mathbb{R}^2)}
\newcommand{\Rtst}{{\mathbb{R}^2}}
\newcommand{\Ztst}{{\mathbb{Z}^2}}
\newcommand{\Zst}{{\mathbb{Z}}}
\newcommand{\Rst}{{\mathbb{R}}}
\newcommand{\Nst}{{\mathbb{N}}}
\newcommand{\bC}{{\mathbb{C}}}
\newcommand{\set}[2]{\big\{ \, #1 \, \big| \, #2 \, \big\}}
\newcommand{\norm}[1]{\lVert#1\rVert}
\newcommand{\supess}{\mathop{\operatorname{esssup}}}

\newcommand{\winzero}{{\varphi_0}}
\newcommand{\winzerol}{{\varphi_{0,\lambda}}}
\newcommand{\curvsystem}{\mathcal{C}_{\Lambda}}
\newcommand{\curvsystemdual}{\widetilde{\mathcal{C}}_{\Lambda}}
\newcommand{\frameop}{S_{\Lambda}}
\newcommand{\frameopz}{S_{\Lambda,0}}
\newcommand{\cframeop}{\widetilde{S}_{\Lambda}}

\newcommand{\hatframeop}{\widehat{\frameop}}
\newcommand{\dil}{\mathrm{Dil}}
\newcommand{\dilA}{\dil_{A}}
\newcommand{\dilAjk}{\dil_{A_{j,k}}}
\newcommand{\dilAjkstar}{\dil_{A_{j,k}^*}}

\newcommand{\Argjk}{\frac{2k\pi}{2^j}}
\newcommand{\volLambda}{\abs{\Lambda}}
\newcommand{\hatframeopg}{\widehat{S_{\Lambda,\gamma}}}
\newcommand{\hatframeopzero}{\widehat{S_{\Lambda,0}}}
\newcommand{\corr}{\Theta}
\newcommand{\rtheta}{(r\cos(\theta),r\sin(\theta))}
\newcommand{\rthetap}{(r'\cos(\theta'),r'\sin(\theta'))}

\newcommand{\cone}{V}
\newcommand{\step}[1]{\medskip \noindent{\bf #1}}
\newcommand{\newdelta}{\tau}
\newcommand{\deltaover}{\tau}
\newcommand{\deltapsi}{\varepsilon}
\newcommand{\moment}{{\varsigma}}
\newcommand{\newa}{{\rho_1}}
\newcommand{\newb}{{\rho_2}}
\newcommand{\asympt}{e^{-\newdelta/a^2}+e^{-\newdelta/b^2}}

\newtheorem{lemma}{Lemma}[section]
\newtheorem{theorem}[lemma]{Theorem}

\newtheorem{prop}[lemma]{Proposition}

\newtheorem{rem}[lemma]{Remark}

\title{Exact and approximate expansions with pure Gaussian wavepackets}
\author{Maarten V. de Hoop}
\address{Department of Mathematics, Purdue University, 150 N. University Street, West Lafayette IN 47907, USA.}
\email{mdehoop@purdue.edu}
\author{Karlheinz Gr\"{o}chenig}
\address{Faculty of Mathematics, University of Vienna \\ 
Oskar-Morgenstern-Platz 1, A-1090 Wien, Austria}
\email{karlheinz.groechenig@univie.ac.at}
\author{Jos\'e Luis Romero}
\address{Faculty of Mathematics, University of Vienna \\ 
Oskar-Morgenstern-Platz 1, A-1090 Wien, Austria}
\email{jose.luis.romero@univie.ac.at}
\date{}
\subjclass[2010]{41A25, 42B05, 42C40}
\keywords{Gaussian wavepackets, wavefront, parabolic scaling}
\thanks{K.G. and J.L.R gratefully acknowledge support from the Austrian 
Science Fund (FWF) [P26273-N25] and [M1586-N25]. M. dH. was supported in 
part by National Science Foundation grant CMG DMS-1025318.}

\begin{document}
\begin{abstract}
We construct frames of wavepackets produced by parabolic dilation,
rotation and translation of (a finite sum of) Gaussians and give asymptotics on the analogue of Daubechies frame
criterion. We show that the coefficients in the corresponding approximate expansion decay fast away from the wavefront
set of the original data.
\end{abstract}
\maketitle

\section{Introduction}

Wavepackets are a powerful  tool for the microlocal analysis of (Fourier
integral) operators and their wave front sets~\cite{cofe78, sm98, sm98-1}. A ``second
dyadic decomposition''~\cite{stein93} with parabolic scaling and
rotations endows them with the necessary  directional
sensitivity  to resolve singularities. Related objects, such as
curvelets and shearlets, incorporate similar features, in particular
the parabolic dilations. Frames of these also facilitate methods to
sparsify Fourier integral operators and resolve the wavefront set; the emphasis is often on  the
applications to  image processing and numerical aspects. See~\cite{cado04,cado05,gula07,shear12,andewe12}
for representative references.

We consider wavepackets produced by parabolic dilation, rotation and translation
of a function $\varphi:\Rtst \to \bC$, given by 
\begin{align}
\label{eq_packet_intro}
\varphi_{j,k,\lambda}(x)
:= 8^{j/2} \varphi (A_{j,k} x -\lambda),
\qquad (j \geq 1, 0 \leq k \leq 2^j-1, \lambda \in \Lambda),
\end{align}
where
\begin{align*}
A_{j,k} := \begin{bmatrix}4^j&0\\0&2^j\end{bmatrix}
\begin{bmatrix}\cos(\frac{2k\pi}{2^j}) & -\sin(\frac{2k\pi}{2^j})\\
\sin(\frac{2k\pi}{2^j})&\cos(\frac{2k\pi}{2^j})
\end{bmatrix},
\end{align*}
and $\Lambda \subseteq \Rtst$ is a lattice. If $\varphi$ is rapidly decaying, the packet
$\varphi_{j,k,\lambda}$ is localized near $A_{j,k}^{-1} \lambda$ and aligned to an angle of
$\frac{-2\pi k}{2^j}$ radians. 
For  the coarse scales we use translates of a second window $\winzerol
:= \winzero(\cdot-\lambda)$, which typically is a symmetric  Gaussian
centered at 0. The complete collection of such wavepackets is then the set 
\begin{align}
\label{eq_curv_system}
\curvsystem := \set{\winzerol}{\lambda \in \Lambda}
\cup \set{\varphi_{j,k,\lambda}}{\lambda \in \Lambda, j \geq 1, 0 \leq k \leq 2^j-1}.
\end{align}
Our  contribution is twofold: First  we consider exact expansions and approximate expansions of
$\LtRt$ functions in terms of these packets using Gaussian
windows. Second we study the wavefront set with respect to a  frame of
wavepackets. 

The motivation for using purely Gaussian windows comes from highly
efficient numerical algorithms \cite{ancate12}, with applications in
reflection and global seismology. The Gaussian is the window of choice,
because it has the best localization in phase-space and yields the
best resolution. The use of wavepackets that consist of pure Gaussians
has several technical advantages because a number of computations can be done explicitly.

Applications include data compression, denoising,
and wavefield recovery from finite data \cite{ACdHa, ACH10,
ATMdH}, and wave propagation and inverse scattering described by
Fourier integral operators associated with canonical graphs. For the
latter application, see \cite{DdH}. Purely Gaussian wave packets are
effective in the parametrix construction of the wave equation. Here
the  the initial data are decomposed with respect to a frame of such
wave packets, because a Gaussian wave packet naturally matches the
initial conditions for a multi-scale Gaussian-beam-like solution
\cite{qiyi10-1}. Gaussian beams are designed to account for the
formation of caustics. The frame property facilitates the derivation
of error estimates
for the parametrix which decay under scale refinement. One can also
generate multi-scale Gaussian-beam-like solutions from the
decomposition of boundary values, which play a role in the development
of inverse scattering
algorithms via wavefield cross correlation that are based on
efficient reverse-time continuation   \cite{OSdH}.

We refer the reader to \cite{ancate12} for a comparison of curvelets
and Gaussian wavepackets and ~\cite{ant13} for the
use of Gaussians as wavelets for the continuous wavelet transform.
Whereas the recent results  on  curvelets and shearlets emphasize
the construction of tight frames with compactly supported windows  or
bandlimited windows~\cite{alcamo04-1,CDo05,gula07,dastte11,rani12,nira12,kikuli12}, we understand 
much less about   the spanning properties and  the frame properties of
 wavepackets with Gaussian windows. For instance, Andersson, Carlsson,
 and Tenorio~ \cite{ancate12} treat the effective \emph{approximation} of bounded compactly-supported functions by
this kind of packets, but they left open the question whether a  full
expansion by a set $\mathcal{C}_\Lambda $ of Gaussian wavepackets represents
every function   in $\LtRt$. Likewise, most of the work on Gaussian
beams~\cite{ACH10,QL10,qiyi10-1,shli09} emphasizes the algorithmic aspects and contents itself
with good numerical approximations. The question of \emph{exact}
representations by a wavepacket expansion is omitted.

Our work complements the numerical  results in \cite{ancate12,qiyi10-1}. 
We will show in Theorem~\ref{th_crit} that the wavepacket system 
$\mathcal{C}_\Lambda$ constitutes a frame
for a large class of Gaussian-like windows and suitable lattices
$\Lambda$.

Using a Gaussian window  adds numerous (technical) difficulties to the
rigorous analysis of wavepacket expansions. On the one hand, the
Gaussian does not have vanishing moments, which is an absolute must in
this analysis.  On the other hand, the Gaussian is neither compactly
supported nor bandlimited; this requires delicate estimates to control
the overlap of dilates and rotations of the Gaussian (see
Proposition~\ref{prop_scale_correlation}). 

 The problems due to the
lack of vanishing moments  can be  circumvented,  as 
with Morlet's wavelets,  by 
considering linear combinations of modulated Gaussians. As a fundamental example we consider windows of the form,
\begin{align}
\label{eq_window_intro}
&\widehat{\varphi}(\xi_1,\xi_2) :=
\sum_{k=1}^N
a_k e^{-\left(\delta_0 \abs{\xi_1-t_k}^2+\delta_k\abs{\xi_2}^2\right)},
\qquad
\xi=(\xi_1,\xi_2) \in \Rtst,
\end{align}
where the points $t_1, \ldots, t_N$, the dilation parameters $\delta_0, \ldots, \delta_N$
and the coefficients $a_1, \ldots, a_N$ are chosen so that
\begin{align}
\label{eq_intro_moments}
\widehat{\varphi}(0)=\partial_{\xi_1} \widehat{\varphi}(0)= \ldots =
\partial^N_{\xi_1} \widehat{\varphi}(0)=0.
\end{align}
For such a choice of $\varphi$, the wavepackets in \eqref{eq_packet_intro} are finite linear combinations
of Gaussians with different eccentricities, scales and centers. It should be noted that this is different
from other Gaussian wavepackets that enforce the vanishing moments through a scale-dependent frequency cut-off
(see for example \cite{andesmuh07}).

The formula in \eqref{eq_window_intro} allows for a number of rather different windows. One possibility is to let
$\widehat{\varphi(\xi)}$ be essentially a translated Gaussian $e^{-\abs{\xi-(0,t_1)}^2}$, where $t_1 >>0$.
This window is then corrected with other Gaussians centered near the origin, so as to
obtain several vanishing moments. If $t_1 >>0$, the coefficients corresponding to these correcting terms can be
numerically neglected. The case where $t_1$ is not large enough to make the correcting terms negligible is also of
practical interest. In this case the window $\varphi$ does not look like a pure Gaussian (see Figure \ref{fig_win}).
Nevertheless, an expansion in terms of the packets in \eqref{eq_window_intro} leads to an expansion in terms of
Gaussians.

\begin{figure}
\centering
\begin{minipage}{0.49\linewidth}\centering
\includegraphics[scale=0.30]{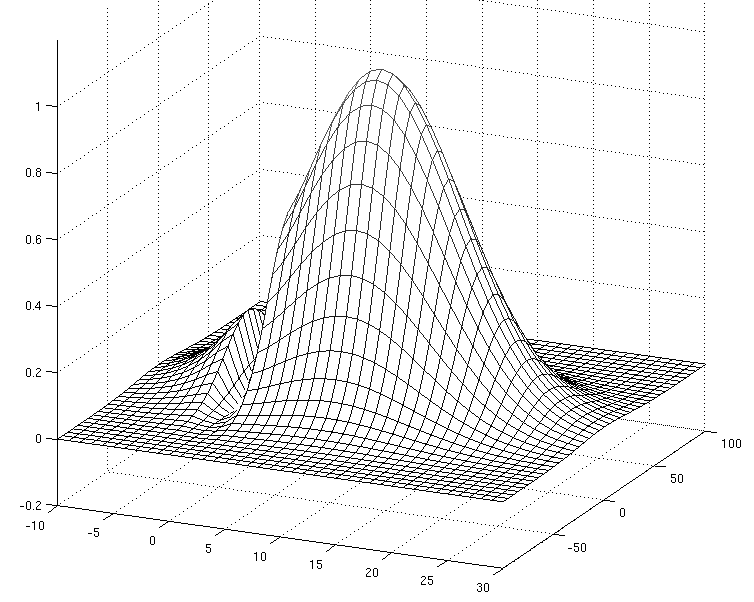}
\end{minipage}
\begin{minipage}{0.49\linewidth}\centering
\includegraphics[scale=0.75]{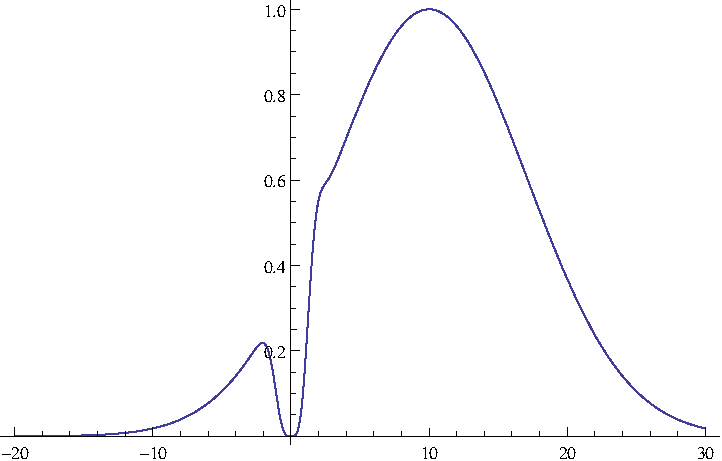}
\end{minipage}
\captionsetup{singlelinecheck=off}    
\caption[.]{A possible choice of $\widehat{\varphi}$ with vanishing moments of order 3.
On the left, a plot of $\widehat{\varphi}$, and on the right a plot of its restriction to the $\xi_1$ axis.
{\tiny \begin{align*}
&\widehat{\varphi(\xi)}=g_1(\xi_1)g_2(\xi_2)
\\
&g_1(t) \approx e^{-\frac{(t-10)^2}{100}} + 0.578 e^{-(t-1)^2} -  3.45205 e^{-(t-0.5)^2}
+ 4.66167 e^{-(t-0.25)^2} - 2.27129 e^{-t^2}
\\
&g_2(t)= e^{-\frac{t^2}{1100}}.
\end{align*}}
\label{fig_win} 
}
\end{figure}

Throughout the article we make the following assumptions on the windows $\winzero, \varphi \in L^2(\Rtst)$.
\begin{align}
\label{eq_assum_cov}
&0 < A \leq 
\abs{\widehat{\winzero}(\xi)}^2 + 
\sum_{j \geq 1} \sum_{k=0}^{2^j-1} \abs{\widehat{\varphi}(B_{j,k} \xi)}^2
\leq B < +\infty,
\mbox{ where } B_{j,k} := (A^*_{j,k})^{-1},
\\
\label{eq_assum_decay_window_1_plus}
&\abs{\widehat{\winzero}(\xi)} \lesssim e^{-\delta\abs{\xi}^2},
\\
\label{eq_assum_decay_window_2_plus}
&\abs{\widehat{\varphi}(\xi)} \lesssim \min(1,\abs{\xi_1}^{\moment}) e^{-\delta\abs{\xi}^2},
\mbox{ for some } \delta>0 \mbox{ and } \moment>2.
\end{align}
The assumption in \eqref{eq_assum_cov} means that the windows are wide enough so as to cover the whole 
frequency plane (see figure \ref{fig_wsum}). The Gaussian wavepackets in 
\eqref{eq_window_intro} clearly satisfy the
decay conditions in \eqref{eq_assum_decay_window_1_plus} and \eqref{eq_assum_decay_window_2_plus} while the vanishing
moments in the direction $\xi_1$ are granted by \eqref{eq_intro_moments}.

Our main result shows that $\mathcal{C}_\Lambda $ is a frame for
$L^2(\mathbb{R}^2)$ for sufficiently  fine lattices $\Lambda $. In
Theorem~\ref{th_frame}   we give asymptotics on the lattice parameters
required for these packets to generate a frame. The proof uses the
method that was developed by  Daubechies  for wavelets~\cite{daubechies90}. We
will  examine the quantities that appear in (the analogue of)
Daubechies' criterion for wavelets, and then  bound  the
correlation between rotations and dilations of the window. A similar
approach (involving fewer technicalities) was recently carried out in \cite{kikuli12} for the
so-called cone-adapted shearlets, in order to obtain compactly supported atoms.
The fact that Daubechies' criterion can be satisfied at
all proves that every $\LtRt$ function can be expanded in terms of
pure Gaussian packets. Thus our main result provides theoretical
support for the accuracy of numerical wavepacket methods
in~\cite{ancate12}. 

As a by-product of using Daubechies' criterion, we obtain a simple
approximate dual frame of the wavepacket system $\mathcal{C}_\Lambda $. The
main term of the associated frame operator can be written as a Fourier
multiplier. Applying the inverse of this Fourier multiplier to the
wavepackets we obtain  an approximate dual frame that can be easily
computed with   a simple  frequency filter. The error estimate
given in Theorem~\ref{th_dual} may be  used to justify when to use the simple
approximate expansion (without knowledge of a precise dual frame).
In the  applications to
geophysics that motivated us, the data are noisy and possess a
limited  numerical accuracy \cite{dode07}.  Therefore, the  explicit,
simple, but only approximate  expansion of the wavefield 
will be sufficient for many purposes in this context.

Our second contribution is the investigation of the wavefront set with
Gaussian wavepackets. We will  study the behavior  of the coefficients $\ip{f}{\varphi_{j,k,\lambda}}$
as $j \rightarrow +\infty$. Since $\mathcal{C}_\Lambda $ is parametrized by a
discrete set and is not invariant under translations and rotations, we
need to select the  parameters $k=k_j,\lambda=\lambda_j$ as a function
of the scale in order  to track a given
phase-space point. Thus the discrete description is more subtle
than the descriptions with a continuous transform in~\cite{folland89,cado05,gr11-8,kula09,gula09}. 
In Theorem~\ref{th_decay_wave} we will  show that if a  point does not belong to the wavefront set of $f$ with respect
to a Sobolev space $H^s$, then the corresponding coefficients decay like $4^{-js}$.
We  also consider the  approximate wavepacket  expansion of Section
3.4  and show that the coefficients decay fast in the scale
parameter $j$, away from the wavefront set of the \emph{original data}.
It is worth pointing out that 
wavefront sets can be characterized with other types of anisotropies besides the parabolic
one, see for example \cite[Theorem 3.29]{folland89}. The parabolic
scaling, however,  is essential for 
the efficient representation of functions with singularities along
smooth curves. 

\begin{figure}
\centering
\includegraphics[scale=0.50]{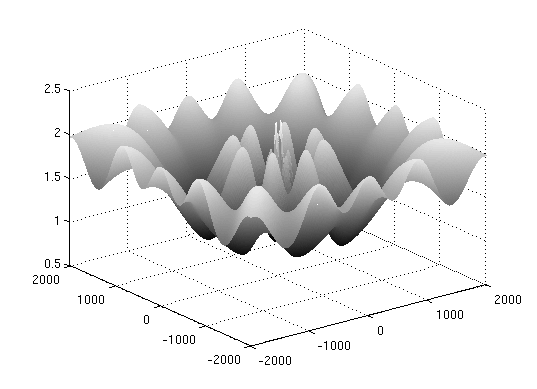}
\captionsetup{singlelinecheck=off}    
\caption[.]{A plot of the quantity in \eqref{eq_assum_cov} for the function
$\varphi$ from Fig. \ref{fig_win} and $\winzero(\xi)=e^{-\frac{\abs{\xi}^2}{10000}}$, showing
$A \approx 0.9503$ and $B \approx 2.2483$.}
\label{fig_wsum}
\end{figure}

\begin{figure}
\centering
\includegraphics[scale=0.4]{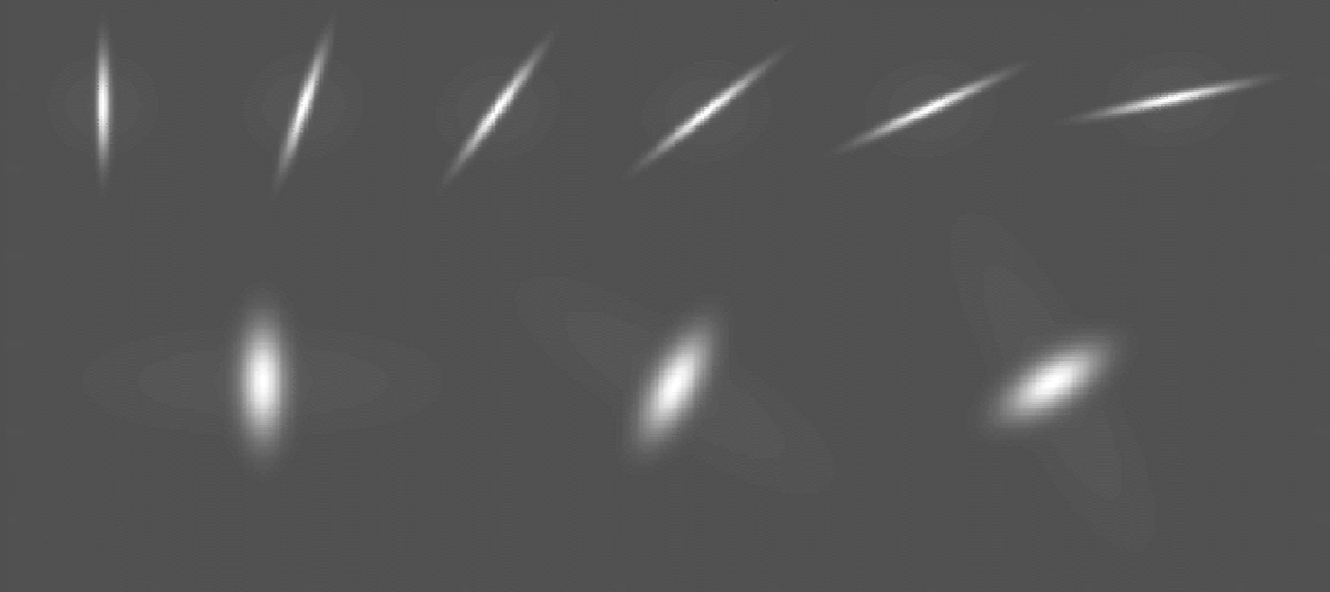}
\captionsetup{singlelinecheck=off}    
\caption[.]{A time-domain display of the wavepackets associated with the window depicted in Figure \ref{fig_win}
(absolute value).}
\end{figure}

\subsection{Related work}
The Gaussian wavepackets that we treat in this article share the geometric features of 
other parabolic wavepackets including curvelets and shearlets. Specifically,
the atoms are concentrated on elongated regions obeying the relation: \emph{width} $\approx$
\emph{length}$^2$ and oscillate across their ridge.
The packets in \eqref{eq_packet_intro}
are produced exactly by applying a parabolic dilation, rotation and translation to a single window function.
In constrast, the classic Gaussian wavepackets from \cite{cofe78,sm98-1,andesmuh07} use an additional frequency cut-off
to enforce vanishing moments. The curvelets from \cite{cado04} are produced by parabolic dilation and rotation of a
slightly different window at each scale. In the case of shearlets, rotations are replaced by the shear maps $S_s(x,y) =
(x+sy,y)$. Wave packets based on tilings of frequency space are
discussed in ~\cite{alcamo04-1,nira12,camoro13,ni13}.

While the differences between the various parabolic wavepackets are
very relevant in practice, 
their asymptotic properties are  similar. The Gaussian wavepackets in
\eqref{eq_packet_intro} are an example of 
\emph{curvelet molecules} from \cite{cade05} and therefore  they
provide the same sparse approximation 
properties as curvelets for functions with discontinuities along smooth edges \cite{cado04}. The notion of curvelet
molecules and the related notion of \emph{shearlet molecules} in
\cite{gula08} are  
in turn particular cases of the notion of \emph{parabolic molecules},
as recently introduced in \cite{grku13}.
In that article it is shown that all expansions with adequate systems of parabolic molecules share similar asymptotic
properties. As a consequence, our results on the decay of frame coefficients away from the wavefront set can be
transferred to other parabolic systems.

Closest to our work is the construction of curvelet-type expansions in
\cite{rani12}. Whereas  we construct frames
by using an analogue of Daubechies criterion for wavelets, the
curvelet-type frames for $\LtRt$ in \cite{rani12}  are 
 obtained  by using a perturbation argument. The main focus of that
 work is the production  of 
compactly-supported atoms, but the results also apply to sums of
modulated Gaussians.  As mentioned in  \cite[Section 6]{rani12}, 
the perturbation argument  yields a frame where each element is a
linear combination of dilated, 
rotated and translated Gaussians. However, the coefficients in that linear
combination depend on the indices of the  particular 
frame element, but the number of terms is proved to be uniformly
bounded.  By  contrast,  the packets in \eqref{eq_packet_intro}
consist exactly of parabolic dilations, rotations 
and translations of a \emph{single} function, which may be taken to be
a fixed  linear combination of Gaussians.

The Gaussian wavepackets with parabolic scaling in \eqref{eq_packet_intro} provide an optimal representation of
functions with singularities along smooth curves \cite{cado04}, sparsify Fourier wave propagators \cite{cade05},
and are particularly
useful in geophysics \cite{dode07}. For other kinds of applications, other geometries are more adequate. For example, in
\cite{deyi07} it is shown that a dictionary of wave-atoms with the scaling ``\emph{wavelength} $\approx$
\emph{diameter}$^2$'' is optimal for the representation of oscillatory patterns (texture).

\subsection{Organization}
The paper is organized as follows. In Section~2 we  derive certain
estimates on rotation-dilation overlaps. This is perhaps the main
technical contribution when working with Gaussian wavepackets. In
Section~3 we  develop a Daubechies-like frame criterion and derive conditions
for when  the
wavepacket system $\mathcal{C}_\Lambda $ is a frame. We also study
the approximate expansions. In Section~4  we study the complement of
the wavefront set and describe it by the decay of the coefficients $\ip{f}{\varphi_{j,k,\lambda}}$
as $j \rightarrow +\infty$.  Finally we show that similar estimates hold
for the coefficients in the approximate expansion related to Daubechies' criterion.

\vspace{ 3 mm}

\textbf{Notation:} In the sequel we use the following notation. We let $1_E$ denote the characteristic function of a
set $E \subseteq \Rtst$ and $\abs{x}$ denote the Euclidean norm of a vector $x \in \Rtst$.
The Fourier transform is normalized as
$\mathcal{F}f(\xi)=\hat{f}(\xi):=\int_\Rtst f(x) e^{-2\pi i \xi x} \, dx$. Translations and dilations of function
$f:\Rtst \to \bC$
are defined as
\begin{align*}
&\dilA f(x) := \abs{\det(A)}^{1/2} f(Ax),
\quad A \in \Rst^{2 \times 2},
\\
&T_y f(x) := f(x-y).
\end{align*}
With this notation, $\varphi_{j,k,\lambda} = \dil_{A_{j,k}} T_\lambda \varphi$. Note also that with
this notation $\dil_{A \cdot B} f = \dil_B \dil_A f$.

Let us further denote the parabolic dilations and the rotations by
\begin{align*}
D_j=\begin{bmatrix}4^j&0\\0&2^j\end{bmatrix},
\qquad
R_\theta=\begin{bmatrix}\cos(\theta) & -\sin(\theta)\\ \sin(\theta)&\cos(\theta)  \end{bmatrix}.
\end{align*}
Hence $A_{j,k} = D_j R_{\Argjk}$ and $B_{j,k}:=(A^*_{j,k})^{-1} = D_j^{-1} R_{\Argjk}$.

If the lattice $\Lambda$ is given by $\Lambda = P \Ztst$, with $P \in \Rst^{2\times 2}$ invertible,
then we denote its \emph{volume} by $\volLambda = \abs{\det{P}}$ and its \emph{dual lattice}
by $\Gamma=(P^{-1})^* \Ztst$.

\section{Estimates for rotation-dilation overlaps}
\label{sec_overlaps}
We now prove certain technical estimates on rotation-dilation overlaps that will be needed throughout the article.
For $f: \Rtst \to \bC$ let the \emph{star-norm} of $f$ be defined by
\begin{align}
\label{eq_star_norm}
\norm{f}_* := \supess_{\xi \in \Rtst} \sum_{j \geq 1} \sum_{k=0}^{2^j-1} \abs{f(B_{j,k}\xi)}.
\end{align}

Let us define the circular sectors,
\begin{align}
\label{eq_deff_Vst}
V_{s,t} := \left\{
(r \cos(\theta), r \sin(\theta)):
\begin{array}{l}
2^s \leq r \leq 2^{s+1}, 0 \leq \theta \leq \pi/2,
\\
2^{-(t+1)} \leq \cos(\theta) \leq 2^{-t}
\end{array}\right\},
\quad (s \in \Zst, t \geq 0).
\end{align}
\begin{figure}
\begin{center}
 \includegraphics[scale=0.8]{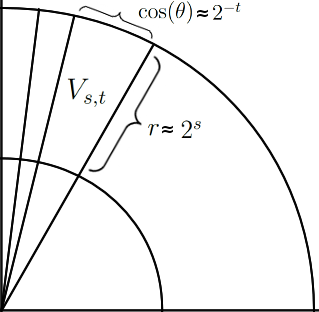} 
\end{center}
\caption{The set $V_{s,t}$ defined in \eqref{eq_deff_Vst}.}
\end{figure}
The next lemma estimates the overlaps of the orbit of each of these sets under parabolic dilations and rotations.
\begin{lemma}
\label{lemma_Vst}
Let $V_{s,t}$ be defined by \eqref{eq_deff_Vst}. Then
$\norm{1_{V_{s,t}}}_* \lesssim 4^t$.
\end{lemma}
\begin{proof}
The effect of the parabolic dilation $D_j$ can be described (on the first quadrant) in the following way.
\begin{align}
\label{eq_parabolic}
&D_j (r\cos(\theta), r\sin(\theta))=(r'\cos(\theta'), r\sin(\theta')),
\\
&r'=\rho(\cos(\theta),j) r,
\\
\label{eq_tan}
&\tan(\theta')=2^{-j}\tan(\theta),
\end{align}
where $0 \leq \theta \leq \pi/2$, $r > 0$, and $\rho(\alpha,j) \geq 0$ is given by
\begin{align}
\rho(\alpha,j)^2 = \alpha^2 4^{2j} + (1-\alpha^2) 2^{2j},
\qquad (\alpha \in [0,1],  j \in \Nst).
\end{align}
We note that $\rho(\alpha,j)$ satisfies the growth estimates,
\begin{align}
&\rho(\alpha, j+1) \geq 2 \rho(\alpha, j),
\\
\label{eq_rho_fast_growth_2}
&\rho(\alpha/2, j+2) \geq 2 \rho(\alpha, j),
\qquad (\alpha \in [0,1],  j \in \Nst).
\end{align}
Note also that for fixed $j \in \Nst$, $\rho(\alpha,j)$ is increasing in $\alpha$.

With this description, we estimate the image of $V_{s,t}$ under  the parabolic dilation $D_j$.
For $s \in \Zst$, $t \geq 0$ and $j \in \Nst$, let us define 
\begin{align*}
W_{j,s,t}:= \left\{ (r\cos(\theta),r\sin(\theta)) :
\begin{array}{c}
2^s \rho(2^{-(t+1)}, j) \leq r \leq  2^{s+1} \rho(2^{-t}, j) \\
0 \leq \theta \leq \pi/2 \min\{1,2^{-j} 4^{t+1}\}
\end{array}\right\}.
\end{align*}

We note that
\begin{align}
\label{eq_AjVst}
D_j V_{s,t} \subseteq W_{j,s,t}.
\end{align}
Indeed, if $\xi=\rtheta \in V_{s,t}$ and $D_j \xi=\rthetap$, then by \eqref{eq_parabolic},
\begin{align*}
r'=\rho(\cos(\theta),j)r \in r [\rho(2^{-(t+1)},j),\rho(2^{-t},j)]
\subseteq [2^s \rho(2^{-(t+1)}, j), 2^{s+1} \rho(2^{-t}, j)]. 
\end{align*}
Since $\xi$ belongs to the first quadrant, so does $D_j \xi$. Hence, $\theta' \in [0,\pi/2)$.
In addition, using \eqref{eq_tan} and the mean value theorem we estimate,
\begin{align*}
\theta' \leq \tan(\theta') = 2^{-j} \tan(\theta)
= 2^{-j} \theta (\cos(\tau))^{-2},
\end{align*}
for some $0 \leq \tau \leq \theta$. Hence, $\cos^2(\tau) \geq \cos^2(\theta) \geq 4^{-t-1}$ and,
$\theta' \leq 2^{-j} 4^{t+1} \pi/2$. This shows that \eqref{eq_AjVst} holds true.

Hence, $D_j$ maps $V_{s,t}$ into $W_{j,s,t}$, a circular sector of angle $\approx 2^{-j} 4^t$. The $2^j$ rotations
$R_{\Argjk} W_{j,s,t}$, $k=0, \ldots, 2^j-1$ have $\lesssim 4^t$ overlaps. Therefore, \eqref{eq_AjVst} 
allows us to estimate for every $j \in \Nst$ and $\xi \in \Rtst$,
\begin{align*}
&\sum_{k=0}^{2^j-1} 1_{V_{s,t}}(B_{j,k} \xi)
= \sum_{k=0}^{2^j-1} 1_{R_{-\Argjk} D_j V_{s,t}}(\xi)
= \sum_{k=0}^{2^j-1} 1_{R_{\Argjk} D_j V_{s,t}}(\xi)
\\
&\quad\leq
\sum_{k=0}^{2^j-1} 1_{R_{\Argjk} W_{j,s,t}}(\xi)
\lesssim 4^t 1_{C_{j,s,t}}(\xi),
\end{align*}
where the set $C_{j,s,t}$ is the annulus
\begin{align}
C_{j,s,t} := \set{\xi \in \Rtst}{2^s \rho(2^{-(t+1)}, j) \leq \abs{\xi} \leq 2^{s+1} \rho(2^{-t}, j)}.
\end{align}
Finally we note that these sets do not overlap too much as $j$ varies. Indeed,
if $C_{j,s,t} \cap C_{j',s,t} \not= \emptyset$ for $j \leq j'$, then
$2^s \rho(2^{-(t+1)}, j') \leq 2^{s+1} \rho(2^{-t}, j)$.
Hence, setting $\alpha := 2^{-t}$, we have that $\rho(\alpha/2, j') \leq 2 \rho(\alpha, j)$.
By \eqref{eq_rho_fast_growth_2} this implies that $j' \leq j+2$.
Hence,
\begin{align*}
\norm{1_{V_{s,t}}}_* \lesssim 4^t 
\sup_{\xi \in \Rtst} \sum_{j \geq 1} 1_{C_{j,s,t}}(\xi) \lesssim 4^t.
\end{align*}
\end{proof}
We now notice the following invariance property of the star-norm.
\begin{lemma}
\label{lemma_sym}
Let $f: \Rtst \to \bC$ and let $S:\Rtst \to \Rtst$ be one of the four symmetries
$(\xi_1,\xi_2) \mapsto (\pm \xi_1, \pm \xi_2)$. Then $\norm{f \circ S}_* = \norm{f}_*$.
\end{lemma}
\begin{proof}
If $S(\xi_1,\xi_2)=(\xi_1,\xi_2)$ or $S(\xi_1,\xi_2)=(-\xi_1,-\xi_2)$, then $S$ commutes with $B_{j,k}$ and the
conclusion is trivial.
If $S(\xi_1,\xi_2)=(-\xi_1,\xi_2)$ or $S(\xi_1,\xi_2)=(\xi_1,-\xi_2)$, then $S$ commutes with $D_j$ and is related to
the rotation of angle
$\theta$ by: $S R_\theta = R_{-\theta} S$. Hence,
\begin{align*}
\norm{f \circ S}_*
&= \supess_{\xi \in \Rtst} \sum_{j \geq 1} \sum_{k=0}^{2^j-1} \abs{f(S D_j^{-1} R_{\Argjk}^{-1} \xi)}
= \supess_{\xi \in \Rtst} \sum_{j \geq 1} \sum_{k=0}^{2^j-1} \abs{f(D_j^{-1} R_{-\Argjk}^{-1} S \xi)}
\\
&= \supess_{\xi \in \Rtst} \sum_{j \geq 1} \sum_{k=0}^{2^j-1} \abs{f(D_j^{-1} R_{-\Argjk}^{-1} \xi)}
= \norm{f}_*,
\end{align*}
where the last equality follows from the fact that $(R_{\Argjk})^{2^j}=I$.
\end{proof}
Finally we derive the main technical time-scale correlation estimate.
\begin{prop}
\label{prop_scale_correlation}
Assume that $f: \Rtst \to \bC$ satisfies
\begin{align}
\label{eq_prop_decay_window}
\abs{f(\xi)} \leq \abs{\xi_1}^{2+\varepsilon} e^{-\deltaover\abs{\xi}}
\end{align} 
for some $\varepsilon, \deltaover>0$. Then
\begin{align*}
\norm{f}_*
= \supess_{\xi \in \Rtst} \sum_{j \geq 1} \sum_{k=0}^{2^j-1} \abs{f(B_{j,k}\xi)}
\lesssim 1.
\end{align*}
Here, the implicit constant depends on $\deltaover$ and $\varepsilon$.
\end{prop}
\begin{rem}
\label{rem_star_norm_finite}
\rm Since the window $\varphi$ generating the wavepackets satisfies \eqref{eq_assum_decay_window_2_plus},
it follows that $\norm{\widehat{\varphi}}_* < +\infty$.
\end{rem}
\begin{proof}[Proof of Proposition \ref{prop_scale_correlation}]
Recall the definition of the sets $V_{s,t}$ in \eqref{eq_deff_Vst} and further define
$V_{s,t}^i:=S^i V_{s,t}$, $i=0,1,2,3$, where $S^i$ are the four symmetries
$(\xi_1,\xi_2) \mapsto (\pm \xi_1, \pm \xi_2)$. From Lemma \ref{lemma_sym},
$\norm{1_{V_{s,t}^i}}_*=\norm{1_{V_{s,t}}}_*$ and hence, by Lemma \ref{lemma_Vst},
$\norm{1_{V_{s,t}^i}}_* \lesssim 4^t$. Since the sets
$\sett{V_{s,t}^i: s \in \Zst, t \in \Nst, i=0,1,2,3}$ cover $\Rtst
\setminus \sett{(\xi_1,\xi_2):\xi_1=0}$ and $f(\xi)=0$ if $\xi_1=0$, we 
have the pointwise estimate
\begin{align*}
\abs{f(\zeta)} \leq \sum_{i=0}^3 \sum_{s \in \Zst} \sum_{t \geq 0}
\sup_{\xi\in V_{s,t}^i} \abs{f(\xi)} 1_{V_{s,t}^i}(\zeta),
\qquad \zeta \in \Rtst,
\end{align*}
and therefore,
\begin{align*}
\norm{f}_* \lesssim \sum_{i=0}^3 \sum_{s \in \Zst}
\sum_{t \geq 0} 4^t \sup_{\xi\in V_{s,t}^i} \abs{f(\xi)}.
\end{align*}
Using \eqref{eq_prop_decay_window} and writing $r := 2^{2+\varepsilon}$ we 
see that
\begin{align}
\sup_{\xi\in V_{s,t}^i} \abs{f(\xi)}
= \sup_{\xi\in V_{s,t}} \abs{f(\xi)}
\lesssim
(2^{s-t+1})^{2+\varepsilon} e^{-\deltaover 2^s}=(1/r)^t r^{s+1} e^{-\deltaover 2^s}.
\end{align}
Since $r>4$,
\begin{align*}
\norm{f}_* \lesssim \sum_{t \geq 0} (4/r)^t \cdot 
\sum_{s \in \Zst} r^{s+1} e^{-\deltaover 2^s} < +\infty.
\end{align*}
\end{proof}

\section{Frame expansions}
\label{sec_frame}
\subsection{Representation of the frame operator}
\label{sec_dec}
Let $\frameop$ be the frame operator associated with the set of wavepackets
$\curvsystem$, given by \eqref{eq_curv_system}
\begin{align*}
\frameop f :=
\sum_{\lambda \in \Lambda}
\ip{f}{T_\lambda \winzero} T_\lambda \winzero
+
\sum_{\lambda \in \Lambda} \sum_{j=1}^{+\infty} \sum_{k=0}^{2^j-1}
\ip{f}{\dilAjk T_\lambda \varphi} \dilAjk T_\lambda \varphi.
\end{align*}

Recall that $\varphi $ and $\varphi _0$ satisfy the hypotheses
\eqref{eq_assum_cov} --  \eqref{eq_assum_decay_window_2_plus}   and
that $\Gamma$ denotes the dual lattice of 
$\Lambda$, precisely, if $\Lambda = P \Ztst$, with $P \in \Rst^{2\times 2}$ 
invertible, then $\Gamma=(P^{-1})^* \Ztst$. 
Using Poisson's summation formula
\begin{align}
\label{eq_psf}
{\volLambda}^{-1}
\sum_{\gamma \in \Gamma} f(x-\gamma)
= \sum_{\lambda \in \Lambda} \widehat{f}(\lambda) e^{2\pi i x \lambda},
\quad (f \in \mathcal{S}(\Rtst)),
\end{align}
we derive a simple description of the action of $\frameop$ in the 
frequency domain that leads to an analogue of
Daubechies' criterion. We now describe this explicitly.

Let us denote $\hatframeop:=\mathcal{F} \frameop \mathcal{F}^{-1}$; i.e., $\hatframeop \widehat{f}
:= \widehat{(\frameop f)}$. Using Poisson's summation formula, $\hatframeop$ can be decomposed in the following way.
\begin{lemma}
\label{lemma_hatframeop}
For $\gamma \in \Gamma$, consider the operator $\hatframeopg: \LtRt \to \LtRt$ defined by
\begin{align}
\label{eq_hatframeopg}
\hatframeopg f (\xi) = f(\xi-\gamma) 
\widehat{\winzero}(\xi) \overline{\widehat{\winzero}(\xi-\gamma)}
+
\sum_{j=1}^{+\infty} \sum_{k=0}^{2^j-1}
f(\xi-A_{j,k}^* \gamma)
\widehat{\varphi}(B_{j,k}\xi) \overline{\widehat{\varphi}(B_{j,k}\xi-\gamma)}.
\end{align}
Then the operator $\hatframeop$ can be written as
\begin{align*}
\hatframeop f = {\volLambda}^{-1} \sum_{\gamma \in \Gamma} \hatframeopg f.
\end{align*}
\end{lemma}
\begin{proof}
Let us consider the operators,
\begin{align*}
S_0 f &:= \sum_{\lambda \in \Lambda} \ip{f}{T_\lambda \winzero} T_\lambda \winzero,
\\
S_* f &:= \sum_{\lambda \in \Lambda} \ip{f}{T_\lambda \varphi} T_\lambda \varphi,
\\
S_{j,k} f
&:= \sum_{\lambda \in \Lambda} \ip{f}{\dilAjk T_\lambda \varphi} \dilAjk T_\lambda \varphi,
\quad (j \geq 1, k=0, \ldots, 2^j-1).
\end{align*}
Hence, $\frameop = S_0 + \sum_{j,k} S_{j,k}$ and $S_{j,k} = \dilAjk S_* \dilAjk^{-1}$.

Let us compute
\begin{align*}
&\mathcal{F} S_0 \mathcal{F}^{-1} f (\xi)
=
\sum_{\lambda \in \Lambda} \ip{\mathcal{F}^{-1} f}{T_\lambda \winzero} e^{-2\pi i \lambda \xi} \widehat{\winzero}(\xi)
\\
&\qquad=
\sum_{\lambda \in \Lambda} \ip{f}{\widehat{\winzero}e^{-2\pi i \lambda \cdot}}
e^{-2\pi i \lambda \xi} \widehat{\winzero}(\xi)
=
\sum_{\lambda \in \Lambda}
\mathcal{F}(f \overline{\widehat{\winzero}})(\lambda) e^{2\pi i \lambda \xi} \widehat{\winzero}(\xi).
\end{align*}
Using \eqref{eq_psf}, this implies that
\begin{align}
\label{eq_lemmaS0}
\mathcal{F} S_0 \mathcal{F}^{-1} f (\xi)
=
\volLambda^{-1}
\sum_{\gamma \in \Gamma} f(\xi-\gamma) \overline{\widehat{\winzero}(\xi-\gamma)} \widehat{\winzero}(\xi).
\end{align}
Similarly,
\begin{align*}
&\mathcal{F} S_* \mathcal{F}^{-1} f (\xi)
=
\volLambda^{-1}
\sum_{\gamma \in \Gamma} f(\xi-\gamma) \overline{\widehat{\varphi}(\xi-\gamma)} \widehat{\varphi}(\xi).
\end{align*}
Noting that $\mathcal{F} S_{j,k} \mathcal{F}^{-1} = 
\dilAjkstar^{-1} \mathcal{F} S_* \mathcal{F}^{-1} \dilAjkstar$,  this yields
\begin{align}
\label{eq_lemmaSjk}
\mathcal{F} S_{j,k} \mathcal{F}^{-1} f(\xi)
=
\volLambda^{-1}
\sum_{\gamma \in \Gamma} f(\xi-A_{j,k}^* \gamma)
\overline{\widehat{\varphi}(B_{j,k}\xi-\gamma)} \widehat{\varphi}(B_{j,k}\xi).
\end{align}
Since $\frameop = S_0 + \sum_{j,k} S_{j,k}$, the conclusion follows from \eqref{eq_lemmaS0}
and \eqref{eq_lemmaSjk}.
\end{proof}
Following the representation in Lemma \ref{lemma_hatframeop}, we define
the time-scale correlation function,
\begin{align}
\label{eq_ts_corr_func}
\corr(\zeta) := \supess_{\xi \in \Rtst}
\left(
\abs{\widehat{\winzero}(\xi)} \abs{\widehat{\winzero}(\xi-\zeta)}+
\sum_{j=1}^{+\infty} \sum_{k=0}^{2^j-1}
\abs{\widehat{\varphi}(B_{j,k}\xi)} \abs{\widehat{\varphi}(B_{j,k}\xi-\zeta)}
\right),
\quad (\zeta \in \Rtst).
\end{align}
We now note that this function bounds the operators appearing in Lemma \ref{lemma_hatframeop}.
\begin{lemma}
\label{lemma_bound_hatframeopg}
The operators in \eqref{eq_hatframeopg} satisfy the bound,
\begin{align*}
\norm{\hatframeopg f}_2 \leq \max\{\corr(\gamma),\corr(-\gamma)\} \norm{f}_2,
\quad
(f \in \LtRt, \gamma \in \Gamma).
\end{align*}
\end{lemma}
\begin{proof}
It is straightforward to see that
$\norm{\hatframeopg f}_1 \leq \corr(-\gamma) \norm{f}_1$ and
$\norm{\hatframeopg f}_\infty \leq \corr(\gamma) \norm{f}_\infty$. Hence the conclusion
follows by interpolation.
\end{proof}
\subsection{A Daubechies-like frame criterion}
We now get the following analogue of Daubechies' criterion for wavelets.
\begin{theorem}
\label{th_crit}
Let
\begin{align}
\label{eq_deff_Delta}
\Delta(\Lambda) := \sum_{\gamma \in \Gamma \setminus \sett{0}} \max\{\corr(\gamma),\corr(-\gamma)\}.
\end{align}
If $\Delta(\Lambda) < A$, then the set of wavepackets
\begin{align*}
\curvsystem =  \set{\winzerol}{\lambda \in \Lambda}
\cup \set{\varphi_{j,k,\lambda}}{\lambda \in \Lambda, j \geq 1, 0 \leq k \leq 2^j-1}
\end{align*}
is a frame of $\LtRt$ with bounds
$\volLambda^{-1} (A - \Delta(\Lambda))$ and $\volLambda^{-1} (B + \Delta(\Lambda))$.
\end{theorem}
\begin{proof}
Recall the decomposition
$\hatframeop f = \volLambda^{-1} \sum_{\gamma \in \Gamma} \hatframeopg f$
in Lemma \ref{lemma_hatframeop} and note that
\begin{align*}
\hatframeopzero f(\xi) = f(\xi)
\left(
\abs{\widehat{\winzero}(\xi)}^2 + \sum_{j=1}^{+\infty} \sum_{k=0}^{2^j-1}
\abs{\widehat{\varphi}(B_{j,k}\xi)}^2
\right).
\end{align*}
The covering condition in \eqref{eq_assum_cov} implies that
\begin{align*}
A \abs{f(\xi)} \leq \abs{\hatframeopzero f(\xi)} \leq B \abs{f(\xi)},
\qquad (f \in \LtRt)
\end{align*}
and therefore
\begin{align}
\label{eq_AB}
A \norm{f}_2 \leq \norm{\hatframeopzero f}_2 \leq B \norm{f}_2,
\qquad (f \in \LtRt).
\end{align}
Using Lemma \ref{lemma_bound_hatframeopg} we get
\begin{align*}
&\norm{\volLambda\hatframeop f - \hatframeopzero f}_2 \leq
\sum_{\gamma \in \Gamma \setminus \sett{0}} \norm{\hatframeopg f}_2 
\\ 
&\qquad \leq \sum_{\gamma \in \Gamma \setminus \sett{0}}
\max\{\corr(\gamma),\corr(-\gamma)\} \norm{f}_2
= \Delta(\Lambda) \norm{f}_2.
\end{align*}
This together with \eqref{eq_AB} this implies that for all $f \in \LtRt$
\begin{align*}
\volLambda^{-1} (A - \Delta(\Lambda)) \norm{f}_2
\leq \norm{\hatframeop f}_2
\leq \volLambda^{-1} (B + \Delta(\Lambda)) \norm{f}_2.
\end{align*}
Applying the last estimate to $\hat{f}$ and using Plancherel's theorem we conclude
that for all $f \in \LtRt$
\begin{align*}
\volLambda^{-1} (A - \Delta(\Lambda)) \norm{f}_2
\leq \norm{\frameop f}_2
\leq \volLambda^{-1} (B + \Delta(\Lambda)) \norm{f}_2,
\end{align*}
as desired.
\end{proof}
\begin{rem} \rm 
The criterion remains true if $\Delta(\Lambda)$ is replaced by the following smaller quantity,
cf. Lemma \ref{lemma_bound_hatframeopg},
\[\sum_{\gamma\not=0} \corr(\gamma)^{1/2}
\corr(-\gamma)^{1/2}.\] 
\end{rem}
\subsection{Asymptotics for the sampling density and the error}
\label{sec_asym}
Using the rotation-dilation estimates from Section \ref{sec_overlaps} we can give asymptotics for the frame criterion
of Theorem \ref{th_crit}. Let the lattice $\Lambda$ be given by $\Lambda=P \Ztst$, with $P$ having singular values
$a,b \geq 0$.
This means that $\Lambda$ is a unitary image of the rectangular lattice 
$a\Zst \times b\Zst$. Motivated by this, 
we call $a,b$ the \emph{lattice parameters} of $\Lambda$. We now estimate the sampling density and reconstruction error
in terms of $a$ and $b$.
\begin{lemma}
\label{lemma_asym}
Under assumptions \eqref{eq_assum_cov}, \eqref{eq_assum_decay_window_1_plus}, \eqref{eq_assum_decay_window_2_plus},
the quantity $\Delta(\Lambda)$ in \eqref{eq_deff_Delta} satisfies the estimate
\begin{align*}
\Delta(\Lambda) \lesssim \asympt,
\end{align*}
for some $0<\newdelta<\delta$ and $0< a,b \leq 1$.
\end{lemma}
\begin{proof}
Let $\xi,\zeta \in \Rtst$. We use repeatedly the quasi triangle inequality:
$\abs{\xi+\zeta}^2 \leq 2(\abs{\xi}^2+\abs{\zeta}^2)$ and use
$\newdelta$ to denote new (smaller) decay parameters that can vary from line to line.
By \eqref{eq_assum_decay_window_1_plus},
\begin{align*}
\abs{\widehat{\winzero }(\xi)}\abs{\widehat{\winzero }(\xi-\zeta)}
\lesssim e^{-\delta(\abs{\xi}^2+\abs{\xi-\zeta}^2)} \leq e^{-\newdelta\abs{\zeta}^2}.
\end{align*}
Likewise, by \eqref{eq_assum_decay_window_2_plus},
\begin{align*}
\abs{\widehat{\varphi }(\xi)}\abs{\widehat{\varphi }(\xi-\zeta)}
\lesssim \abs{\xi_1}^{\moment}
e^{-\delta(\abs{\xi}^2+\abs{\xi-\zeta}^2)}.
\end{align*}
Since $\abs{\xi}^2+\abs{\xi-\zeta}^2 \gtrsim \max\{\abs{\xi}^2, \abs{\zeta}^2\}$, we have that
$\abs{\xi}^2+\abs{\xi-\zeta}^2 \gtrsim \abs{\xi}^2+\abs{\zeta}^2$. Hence,
\begin{align*}
\abs{\widehat{\varphi }(\xi)}\abs{\widehat{\varphi }(\xi-\zeta)}
\lesssim \abs{\xi_1}^{\moment} e^{-\newdelta \abs{\xi}^2} e^{-\newdelta \abs{\zeta}^2}
= \Phi(\xi) e^{-\newdelta \abs{\zeta}^2},
\end{align*}
if we let $\Phi(\xi) := \abs{\xi_1}^{\moment} e^{-\newdelta \abs{\xi}^2}$.
This implies that the time-scale correlation function in \eqref{eq_ts_corr_func} satisfies
\begin{align*}
\corr(\zeta) &\lesssim
e^{-\delta\abs{\zeta}^2} +
\supess_{\xi\in\Rtst}
\sum_{j=1}^{+\infty} \sum_{k=0}^{2^j-1}
\abs{\widehat{\varphi}(B_{j,k}\xi)} \abs{\widehat{\varphi}(B_{j,k}\xi-\zeta)}
\\
&\leq
e^{-\newdelta\abs{\zeta}^2} +
e^{-\newdelta\abs{\zeta}^2}
\supess_{\xi\in\Rtst}
\sum_{j=1}^{+\infty} \sum_{k=0}^{2^j-1}
\Phi (B_{j,k}\xi)
\\
&= e^{-\newdelta\abs{\zeta}^2}(1+\norm{\Phi}_*).
\end{align*}
Using Proposition \ref{prop_scale_correlation} we conclude that
\begin{align}
\label{eq_bound_corr}
\corr(\zeta) \lesssim e^{-\newdelta\abs{\zeta}^2}.
\end{align}
Recall that $\Lambda=P\Ztst$ and let us write $P=UD$ with $D=diag(a,b)$, $U$ unitary and $a,b \geq 0$.
Then $\Gamma=UD^{-1} \Ztst$ and we can estimate
\begin{align*}
\Delta(\Lambda)&\lesssim
\sum_{(k,j) \in \Ztst \setminus\sett{0}} e^{-\newdelta\abs{U D^{-1}(k,j)}^2}
\quad=
\sum_{(k,j) \in \Ztst \setminus\sett{0}} e^{-\newdelta\abs{D^{-1}(k,j)}^2}
\\
&\quad \leq
\sum_{(k,j) \in \Ztst \setminus\sett{0}} e^{-\newdelta k^2/a^2} e^{-\newdelta j^2/b^2}
\lesssim
\sum_{(k,j) \in \Nst^2 \setminus\sett{0}} e^{-\newdelta k/a^2} e^{-\newdelta j/b^2}.
\end{align*}
Summing the geometric series gives
\begin{align*}
\Delta(\Lambda) & \lesssim (1-e^{-\newdelta/a^2})^{-1}(1-e^{-\newdelta/b^2})^{-1}-1
=
\frac{1-(1-e^{-\newdelta/a^2})(1-e^{-\newdelta/b^2})}{(1-e^{-\newdelta/a^2})(1-e^{-\newdelta/b^2})}
\\
&\lesssim 1-(1-e^{-\newdelta/a^2})(1-e^{-\newdelta/b^2})
\\
&= e^{-\newdelta/a^2} + e^{-\newdelta/b^2} - e^{-\newdelta/a^2}e^{-\newdelta/b^2}
\\
&\leq
e^{-\newdelta/a^2} + e^{-\newdelta/b^2}.
\end{align*}
\end{proof}
As an immediate consequence of Theorem \ref{th_crit} and Lemma \ref{lemma_asym} we have the following
result.
\begin{theorem}
\label{th_frame}
Assume that \eqref{eq_assum_cov}, \eqref{eq_assum_decay_window_1_plus},
\eqref{eq_assum_decay_window_2_plus} hold. Then
there are constants $C,\newdelta>0$ depending on $\delta$, $\moment$ and
$\varepsilon$ such that $\curvsystem$ is a frame whenever
\begin{align}
\label{eq_sampling_condition}
\kappa:=\asympt \leq C A.
\end{align}
Moreover, in this case, $\volLambda^{-1}(A - C^{-1}\kappa)$ and $\volLambda^{-1}(B+C^{-1}\kappa)$
are frame bounds for $\curvsystem$.
\end{theorem}
In particular, for every lattice $\Lambda$ with sufficiently large density (i.e. $a,b \ll 1$), the system $\curvsystem$
is a frame of $\LtRt$ and yields the following expansions
\begin{align}
\label{eq_frame_expansion_1}
f &= \sum_{\lambda \in \Lambda}
\ip{f}{\frameop^{-1} \winzerol} \winzerol
+
\sum_{\lambda \in \Lambda} \sum_{j=1}^{+\infty} \sum_{k=0}^{2^j-1}
\ip{f}{\frameop^{-1} \varphi_{j,k,\lambda}} \varphi_{j,k,\lambda}
\\
\label{eq_frame_expansion_2}
&= \sum_{\lambda \in \Lambda}
\ip{f}{\winzerol} \frameop^{-1} \winzerol
+
\sum_{\lambda \in \Lambda} \sum_{j=1}^{+\infty} \sum_{k=0}^{2^j-1}
\ip{f}{\varphi_{j,k,\lambda}} \frameop^{-1} \varphi_{j,k,\lambda}.
\end{align}
As the lattice parameters $a,b \longrightarrow 0^+$, the ratio of the 
frame bounds tends to $B/A$.

\begin{rem} \rm 
If the window $\varphi$ is given by \eqref{eq_window_intro}, then \eqref{eq_frame_expansion_1} provides a formal
argument for the representability of $L^2$ functions in terms of the Gaussian wavepackets discussed in \cite{ancate12},
which can be then computed with the algorithms discussed there.
\end{rem}

\begin{rem} \rm 
The system $\curvsystem$ is a family of \emph{curvelet molecules} in the sense of \cite{cade05} and hence the expansion
in \eqref{eq_frame_expansion_2} possesses the same asymptotic sparse approximation as curvelets \cite{cado04}.
\end{rem}

\begin{rem} \rm 
As the proof of Proposition \ref{prop_scale_correlation} shows, its
conclusion holds under much weaker conditions  than  those in \eqref{eq_assum_decay_window_1_plus},
\eqref{eq_assum_decay_window_2_plus}, for example polynomial decay.
\end{rem}
\begin{rem} \rm 
Theorem \ref{th_frame} is analogous to a recent result for shearlets \cite{kikuli12}. The factor
$\abs{\xi_1}^{2+\varepsilon}$  in \eqref{eq_assum_decay_window_2_plus} should be compared with the factor
$\abs{\xi_1}^{3+\varepsilon}$ in \cite{kikuli12}.
\end{rem}

\subsection{Approximate reconstruction}
Since the frame criterion in Theorem \ref{th_crit} consists in showing that
the frame operator $\frameop$ is close to $\frameopz$, in practice one often uses
$\frameopz$ to construct an approximate dual frame. This approximation is particularly convenient because
$\frameopz$ is simply the Fourier multiplier with symbol
\begin{align*}
m(\xi) := \abs{\widehat{\winzero}(\xi)}^2 + \sum_{j \geq 1} \sum_{k=0}^{2^j-1} 
\abs{\widehat{\varphi}(B_{j,k} \xi)}^2,
\qquad (\xi \in \Rtst),
\end{align*}
(cf. Lemma \ref{lemma_hatframeop}). (Recall that the Fourier multiplier with symbol $m \in L^\infty(\Rtst)$
is the operator $M_m: L^2(\Rtst) \to L^2(\Rtst)$ given by $\widehat{(M_m f)}(\xi) = m(\xi) \widehat{f}(\xi)$.)

Hence, the approximate dual system is just the image of $\curvsystem$
by the Fourier multiplier with symbol $1/m$. It then readily follows from the proof of Theorem \ref{th_decay_wave}
and the estimates of Section \ref{sec_asym} that this gives a reconstruction error of at most $\Delta(\Lambda)/A
\lesssim
\asympt$, with $a,b$ the lattice parameters.

By slightly relaxing this bound one can gain certain liberty in the design of the approximate dual system, which can be
exploited to improve the function properties. If one considers a second window $\psi \in \LtRt$, then the frame-type
operator $\cframeop: \LtRt \to \LtRt$,
\begin{align*}
\cframeop f := 
\sum_{\lambda \in \Lambda}
\ip{f}{T_\lambda \winzero} T_\lambda \winzero
+
\sum_{\lambda \in \Lambda} \sum_{j=1}^{+\infty} \sum_{k=0}^{2^j-1}
\ip{f}{\dilAjk T_\lambda \psi} \dilAjk T_\lambda \varphi
\end{align*}
can be decomposed in a way completely analogous to the one in Lemma \ref{lemma_hatframeop}. (This can be seen from
the proof of Lemma \ref{lemma_hatframeop}, or obtained a posteriori by polarization).
The Daubechies-like criterion in this case is expressed in terms of lower and upper bounds for
\begin{align}
\label{eq_approx_m}
\widetilde{m}(\xi) := \abs{\widehat{\winzero}(\xi)}^2 + 
\sum_{j \geq 1} \sum_{k=0}^{2^j-1} 
\widehat{\varphi}(B_{j,k} \xi)\overline{\widehat{\psi}(B_{j,k} \xi)},
\end{align}
and upper bounds for
\begin{align*}
\widetilde{\Delta}(\Lambda) := \sum_{\gamma \in \Gamma \setminus \sett{0}}
\max\{\widetilde{\corr}_1(\gamma),\widetilde{\corr}_2(\gamma)\},
\end{align*}
where
\begin{align*}
\widetilde{\corr}_1(\zeta) &:= \supess_{\xi \in \Rtst}
\left(
\abs{\widehat{\winzero}(\xi)} \abs{\widehat{\winzero}(\xi-\zeta)}+
\sum_{j=1}^{+\infty} \sum_{k=0}^{2^j-1}
\abs{\widehat{\varphi}(B_{j,k}\xi)} |\hat{\psi}(B_{j,k}\xi-\zeta)|
\right),
\\
\widetilde{\corr}_2(\zeta) &:= \supess_{\xi \in \Rtst}
\left(
\abs{\widehat{\winzero}(\xi)} \abs{\widehat{\winzero}(\xi+\zeta)}+
\sum_{j=1}^{+\infty} \sum_{k=0}^{2^j-1}
\abs{\widehat{\varphi}(B_{j,k}\xi+\zeta)} |\hat{\psi}(B_{j,k}\xi)|
\right).
\end{align*}
In our case it makes sense to choose $\psi$ as a small modification of $\varphi$ that enforces some extra frequency
cancellation.
Let $\deltapsi>0$. According
to our assumptions, the set $E_\deltapsi := \set{\xi \in \Rtst}{\abs{\widehat{\varphi}(\xi)} > \deltapsi}$ is an open
bounded set
at positive distance from the origin. Let $\eta:\Rtst \to [0,1]$ be a smooth function with compact support not
containing the origin such that $\eta \equiv 1$ on $E_\deltapsi$. Let $\psi \in \LtRt$ be given by
\begin{align*}
\widehat{\psi} := \eta \widehat{\varphi}.
\end{align*}
The following lemma shows that this approximation is compatible with the quantities in
Daubechies' criterion.
\begin{lemma}
\label{lemma_comp}
The following estimates hold.
\begin{align*}
&\abs{m(\xi) - \widetilde{m}(\xi)} \leq \norm{\widehat{\varphi}}_* \deltapsi \approx \deltapsi,
\\
&0 \leq \widetilde{\Delta}(\Lambda) \leq \Delta(\Lambda).
\end{align*}
\end{lemma}
\begin{proof}
Recall that by Remark \ref{rem_star_norm_finite}
$\norm{\widehat{\varphi}}_* <+\infty$
and note that, by definition, $(1-\eta(\xi))\abs{\widehat{\varphi}(\xi)} \leq \deltapsi$. Hence,
we simply estimate
\begin{align*}
\abs{m(\xi) - \widetilde{m}(\xi)}
&=\abs{
\sum_{j \geq 1} \sum_{k=0}^{2^j-1} \widehat{\varphi}(B_{j,k} \xi)
(1-\eta(B_{j,k} \xi))\overline{\widehat{\varphi}(B_{j,k} \xi)}}
\\
&\leq \deltapsi \sum_{j \geq 1} \sum_{k=0}^{2^j-1} \abs{\widehat{\varphi}(B_{j,k} \xi)}
= \deltapsi \norm{\widehat{\varphi}}_*.
\end{align*}
The bound for $\widetilde{\Delta}(\Lambda)$ follows from 
the fact that $|\widehat{\psi}| = \abs{\widehat{\varphi}} \eta \leq \abs{\widehat{\varphi}}$
and consequently
$\widetilde{\corr}_1(\zeta) \leq \corr(\zeta)$ 
and
$\widetilde{\corr}_2(\zeta) \leq \corr(-\zeta)$.
\end{proof}
As a consequence of these estimates, if $\deltapsi \norm{\widehat{\varphi}}_* < A$, then
\begin{align}
\label{eq_approx_covering}
0< \widetilde{A} := A - \deltapsi \norm{\widehat{\varphi}}_* \leq \widetilde{m}(\xi)
\leq \widetilde{B} := B + \deltapsi \norm{\widehat{\varphi}}_* < +\infty,
\end{align}
and the Fourier multiplier with symbol $\widetilde{m}$ is invertible on $\LtRt$.
We can then consider the following system
\begin{align*}
&\curvsystemdual := \set{\widetilde{\winzerol}}{\lambda \in \Lambda} \cup
\set{\widetilde{\varphi}_{j,k,\lambda}}{j \geq 1, 1 \leq k \leq 2^j-1, \lambda \in \Lambda},
\\
&\widehat{\widetilde{\winzerol}}(\xi) := \frac{1}{\widetilde{m}(\xi)} \widehat{T_\lambda \winzero}(\xi),
\\
&\widehat{\widetilde{\varphi}_{j,k,\lambda}}(\xi) := \frac{1}{\widetilde{m}(\xi)} \widehat{\dilAjk T_\lambda \psi}(\xi).
\end{align*}
This provides the following approximate reconstruction.
\begin{theorem}
\label{th_dual}
Let $\deltapsi \norm{\widehat{\varphi}}_* < A$ and $\Delta(\Lambda) < A - \deltapsi
\norm{\widehat{\varphi}}_*$.
Then  every $f \in \LtRt$ admits the following approximate expansion.
\begin{align}
\label{eq_app_exp}
\widetilde{f} &:= \volLambda\left(
\sum_{\lambda \in \Lambda} \ip{f}{\widetilde{\winzerol}} \winzerol
+ 
\sum_{j \geq 1} \sum_{k=0}^{2^j-1} \sum_{\lambda \in \Lambda}
\ip{f}{\widetilde{\varphi}_{j,k,\lambda}} \varphi_{j,k,\lambda}\right),
\end{align}
with $\norm{f-\widetilde{f}}_2 \leq
\frac{\Delta(\Lambda)}{A-\deltapsi \norm{\widehat{\varphi}}_*}\norm{f}_2
\lesssim (\asympt)\norm{f}_2$.
\end{theorem}
\begin{rem} \rm 
\label{rem_mult}
Note that the coefficients in \eqref{eq_app_exp} are given by
\begin{align*}
\ip{f}{\widetilde{\varphi}_{j,k,\lambda}}=\ip{M_{1/\widetilde{m}} f}{\psi_{j,k,\lambda}},
\end{align*}
where $M_{1/\widetilde{m}}$ is the Fourier multiplier with symbol
$1/\widetilde{m}$,  i.e.,  $\widehat{(M_{1/\widetilde{m}} f)}(\xi) =
1/\widetilde{m}(\xi) \widehat{f}(\xi)$,  
and the packets $\psi_{j,k,\lambda}=\dilAjk T_\lambda \psi$ have rotation-dilation structure.
That is why in practice we prefer the expansion in \eqref{eq_app_exp} to the ``abstract''
expansion in \eqref{eq_frame_expansion_1}.
\end{rem}
\begin{proof}[Proof of Theorem \ref{th_dual}]
Let us denote by $M_{\widetilde{m}}$ the Fourier multiplier with symbol $\widetilde{m}$.
As in the proof of Theorem \ref{th_crit}, for every $f \in \LtRt$
\begin{align*}
\norm{\volLambda \cframeop f-M_{\widetilde{m}}f}_2 \leq \widetilde{\Delta}(\Lambda) \norm{f}_2. 
\end{align*}
Applying this to $M_{\widetilde{m}}^{-1}f$ and using Lemma \ref{lemma_comp} and \eqref{eq_approx_covering}
we obtain
\begin{align*}
\norm{\volLambda \cframeop M_{\widetilde{m}}^{-1}f-f}_2 \leq
\Delta(\Lambda) \norm{M_{\widetilde{m}}^{-1}f}_2 
\leq \frac{\Delta(\Lambda)}{A - \deltapsi \norm{\widehat{\varphi}}_*}\norm{f}_2.
\end{align*}
Finally note that
$\widetilde{f}=\volLambda\cframeop M_{1/\widetilde{m}}f=\volLambda\cframeop M_{\widetilde{m}}^{-1}f$.
\end{proof}
\section{Decay estimates away from the wavefront set}
It is well-know that the continuous transform associated with parabolic representations detects the wavefront set of a
distribution, provided that the generating window has sufficiently many vanishing moments \cite{cado05, kula09,
gr11-8} (see also \cite{ge90}). In the case of wavepacket coefficients, the situation is more technical since one needs
to approximate a given phase-space point by discrete parameters (see \cite{jopiteto12} for a similar problem related to
Gabor expansions).
\subsection{Wavefront sets with wavepackets}
\label{sec_wavefront}
Let $(x_0,\theta_0) \in \Rtst \times [0,2\pi)$. For each scale $j \geq 1$, we select a wavepacket of scale $j$ that
is localized near $x_0$ and is approximately aligned to $\theta_0$. Recall that the wavepacket $\varphi_{j,k,\lambda}$
is
localized near $A_{j,k}^{-1} \lambda$ and aligned to an angle of $\frac{-2\pi k}{2^j}$ radians. We choose the
parameters $k_j=k_j(x_0,\theta_0) \in \sett{0, \ldots, 2^j-1}$ and $\lambda_j=\lambda_j(x_0,\theta_0) \in \Lambda$
that give the best approximation of $\theta_0$ and $x_0$ respectively. More precisely, for every
$j \geq 1$ we choose $(k_j, \lambda_j)$ such that
\begin{align}
\label{eq_best_kj}
&\frac{2\pi k_j}{2^j} \leq 2\pi-\theta_0 \leq \frac{2\pi (k_j+1)}{2^j},
\\
\label{eq_best_lambdaj}
&\abs{A_{j,k_j} x_0 - \lambda_j} \leq L_\Lambda,
\end{align}
where $L_\Lambda$ is a constant that depends on the lattice $\Lambda$
(namely the diameter of its fundamental domain).
For certain $(x_0,\theta_0)$ there is more than one possible choice of $(k_j, \lambda_j)$. Any of these will be
adequate. The resulting sequence $\sett{(k_j, \lambda_j): j \geq 1}$ will be called
the \emph{grid parameters} related to $(x_0,\theta_0)$.

The objective of this section is to show that when $(x_0,\theta_0)$ does not belong to the wavefront set of $f$,
then the numbers $\ip{f}{\varphi_{j,k_j,\lambda_j}}$ decay fast as $j \rightarrow +\infty$.

We denote by $\cone_\varepsilon$ the cone of width $\varepsilon$ aligned along the $\xi_1$-axis
\begin{align}
\label{eq_def_cone}
\cone_\varepsilon := \set{\xi \in \Rtst}{\abs{\xi_2} \leq \varepsilon \abs{\xi_1}}.
\end{align}
We say that a distribution $f \in \mathcal{S}'(\Rtst)$ belongs to the microlocal Sobolev space $H^s(x_0,\theta_0)$ if
there exist $\varepsilon>0$ and a smooth compactly-supported function $\eta$ with $\eta \equiv 1$ near $x_0$ and
such that
\begin{align}
\label{eq_micro_cone}
\int_{R_{\theta_0} \cone_\varepsilon} \abs{\widehat{\eta f}(\xi)}^2 \abs{\xi}^{2s} \,d\xi <+\infty.
\end{align}
Since we work with windows with a limited number of vanishing moments, we further consider
the space $H_M^s(x_0,\theta_0)$ consisting of distributions in $H^s(x_0,\theta_0)$ that satisfy, in
addition to \eqref{eq_micro_cone},
\begin{align}
\label{eq_big_bound_micro}
\int_{\Rtst} \abs{\widehat{\eta f}(\xi)}^2 (1+\abs{\xi})^{-2M} \,d\xi <+\infty.
\end{align}
Note that $H^s(x_0,\theta_0) = \bigcup_{M>0} H_M^s(x_0,\theta_0)$.

We now state the main microlocal estimate. Although we are mainly interested in the case
of $L^2$-functions (i.e. $M=0$), the estimate is valid for distributions
provided that the window $\varphi$ belongs to the Schwartz class.
\begin{theorem}
\label{th_decay_wave}
Let $s, M \geq 0$. Assume that $\varphi \in \mathcal{S}(\Rtst)$ and that
$\moment> 2s+2M-1/2$.

Let $(x_0,\theta_0) \in \Rtst \times [0,2\pi)$, and let
$\sett{(k_j, \lambda_j): j\geq 1}$ be the corresponding grid parameters
given by \eqref{eq_best_kj} and \eqref{eq_best_lambdaj}. Let $f \in H_M^s(x_0,\theta_0)$, then
\begin{align*}
\sum_{j \geq 1} \abs{\ip{f}{\varphi_{j,k_j,\lambda_j}}}^2 4^{2js} < +\infty.
\end{align*}
\end{theorem}

Since the parameters $(k_j,\lambda_j)$ do not exactly capture the pair $(x_0,\theta_0)$
but only provide the best approximation at scale $j$, the proof of Theorem \ref{th_decay_wave}
requires us to control a certain scale-dependent approximation. For this reason we need a microlocal
version of Lemma \ref{lemma_Vst}, as formulated by the following technical statement.
\begin{lemma}
\label{lemma_rot_overlaps}
For $r, t \in \Zst$ consider the intervals,
\begin{align}
\label{eq_sets_Wrt}
W_{r,t} := \set{\xi \in \Rtst}{4^{r-1} \leq \abs{\xi_1} \leq 4^r, 2^{t-1} \leq \abs{\xi_2} \leq 2^t}.
\end{align}
For each $j \geq 1$, let $\theta_j \in [-2^{-j} 2\pi,2^{-j} 2\pi]$ be arbitrary. Then
\begin{align*}
\sup_{r,t \in \Zst} \sup_{\xi \in \Rtst} \sum_{j \geq 1} 1_{R_{\theta_j}D_j W_{r,t}} (\xi) < +\infty.
\end{align*}
That is, each of the families of sets $\sett{R_{\theta_j}D_j W_{r,t}: j\geq 1}$ has a number of overlaps that is bounded
independently of $r,t$.
\end{lemma}
\begin{proof}
Let $r,t \in \Zst$ and suppose that $(R_{\theta_j}D_j W_{r,t}) \cap (R_{\theta_{j'}}D_{j'} W_{r,t}) \not= \emptyset$
for some $1 \leq j \leq j'$. Let $h:= j'-j$. We want to find an absolute bound on $h$.

It follows that
there exists some $\xi \in (R_{\theta_j-\theta_{j'}}D_j W_{r,t}) \cap (D_h D_{j} W_{r,t})$.
Since $\xi \in D_h D_{j} W_{r,t}$, we have
\begin{align}
\label{eq_x1}
&\abs{\xi_1} \approx 4^{h+j+r}=2^{2h+2j+2r},
\\
\label{eq_x2}
&\abs{\xi_2} \approx 2^{h+j+t}.
\end{align}
Similarly, since $\xi \in R_{\theta_j-\theta_{j'}}D_j W_{r,t}$
\begin{align*}
&\xi_1 = \cos(\theta_j-\theta_{j'}) \zeta_1 + \sin(\theta_j-\theta_{j'}) \zeta_2,
\\
&\xi_2 = -\sin (\theta_j-\theta_{j'}) \zeta_1 + \cos(\theta_j-\theta_{j'}) \zeta_2,
\end{align*}
for some $\zeta=(\zeta_1,\zeta_2)$ such that $\abs{\zeta_1} \approx 4^{j+r}=2^{2j+2r}$ and
$\abs{\zeta_2} \approx 2^{j+t}$. Since
$\abs{\theta_j-\theta_{j'}} \lesssim 2^{-j} + 2^{-j'} \approx 2^{-j}$, 
we have that $\abs{\sin(\theta_j-\theta_{j'})} \lesssim 2^{-j}$.
Using this and that $\abs{\cos(\theta_j-\theta_{j'})} \leq 1$ we estimate
\begin{align}
&\abs{\xi_1} \lesssim 2^{2j+2r} + 2^{-j} 2^{j+t} = 2^{2j+2r} + 2^t,
\\
&\abs{\xi_2} \lesssim 2^{-j} 2^{2j+2r} + 2^{j+t} = 2^{j+2r} + 2^{j+t}.
\end{align}
Comparing these equations with \eqref{eq_x1} and \eqref{eq_x2} we get
\begin{align*}
&2^{2h+2j+2r} \lesssim 2^{2j+2r} + 2^t,
\\
&2^{h+j+t} \lesssim 2^{j+2r} + 2^{j+t}.
\end{align*}
Hence, there exists an absolute constant $C>0$ such that
\begin{align*}
&2h+2j+2r \leq \max\{2j+2r+C, t+C\},
\\
&h+j+t \leq \max\{j+2r+C, j+t+C\}.
\end{align*}
If $h=j'-j \leq C$, we are done. If on the contrary $h>C$, then the previous estimates reduce to
\begin{align}
\label{eq_h_1}
&2h+2j+2r \leq t+C,
\\
\label{eq_h_2}
&h+t \leq 2r+C.
\end{align}
Since $h,j \geq 0$ we get from \eqref{eq_h_1} that
$2r \leq 2h+2j+2r \leq t+C$. Plugging this into \eqref{eq_h_2} gives
$h+t \leq t+2C$. Hence $h \leq 2C$. This completes the proof.
\end{proof}

\begin{proof}[Proof of Theorem \ref{th_decay_wave}]
For convenience, let us define $\theta_j := \frac{2 \pi k_j}{2^j}$ and
$\varphi_j(x) := \dil_{D_j} \varphi (x) = 8^{j/2} \varphi(4^jx_1, 2^jx_2)$.
Hence $\varphi_{j,k_j,\lambda_j} = \dil_{R_{\theta_j}} \dil_{D_j} T_{\lambda_j} \varphi = 
\dil_{R_{\theta_j}} T_{D_j^{-1} \lambda_j} \varphi_j$. We now further define
\begin{align*}
g_j :=  T_{-D_j^{-1} \lambda_j} \dil_{R_{-\theta_j}} f,
\quad (j\geq 1).
\end{align*}
Consequently,
\begin{align*}
\ip{f}{\varphi_{j,k_j,\lambda_j}} = \ip{g_j}{\varphi_j}.
\end{align*}
Since $f \in H_M^{s}(x_0,\theta_0)$ there exist a smooth compactly-supported function $\eta$ and $\gamma>0$
such that $\eta \equiv 1$ on the Euclidean ball $B_\gamma(x_0)$ and \eqref{eq_micro_cone} and \eqref{eq_big_bound_micro}
hold. Let us consider the functions,
\begin{align*}
g^1_j &:=  T_{-D_j^{-1} \lambda_j} \dil_{R_{-\theta_j}} (\eta f), \quad (j\geq 1),
\\
g^2_j &:=  T_{-D_j^{-1} \lambda_j} \dil_{R_{-\theta_j}} ((1-\eta) f), \quad (j\geq 1).
\end{align*}
Hence,
\begin{align*}
\ip{f}{\varphi_{j,k_j,\lambda_j}} = \ip{g^1_j}{\varphi_j} + \ip{g^2_j}{\varphi_j}.
\end{align*}
The proof consists in showing that the last two terms have the desired decay as $j \rightarrow +\infty$.

\step{Step 1}. We show that
$\sum_{j \geq 1} \abs{\ip{g^2_j}{\varphi_j}}^2 4^{2js} < +\infty$.

Since $\eta \equiv 1$ on $B_\gamma(x_0)$,
\begin{align*}
g^2_j = T_{-D_j^{-1} \lambda_j} \dil_{R_{-\theta_j}} ((1-\eta) f) \equiv 0
\mbox{ on } R_{\theta_j} B_\gamma(x_0)  - D^{-1}_j \lambda_j
= B_\gamma(R_{\theta_j} (x_0 - A_{j,k_j}^{-1} \lambda_j)).
\end{align*}
Using \eqref{eq_best_lambdaj} we estimate,
\begin{align*}
&\norm{R_{\theta_j}(x_0 - A_{j,k_j}^{-1} \lambda_j)} = 
\norm{x_0 - A_{j,k_j}^{-1} \lambda_j}
= \norm{A_{j,k_j}^{-1}(A_{j,k_j} x_0 -  \lambda_j)}
\\
&\qquad 
\leq 2^{-j} \norm{A_{j,k_j} x_0 -  \lambda_j}
\leq 2^{-j} L_\Lambda.
\end{align*}
For $j \gg 0$, $2^{-j} L_\Lambda \leq \gamma/2$ and
therefore
$B_{\gamma/2}(0) \subseteq B_\gamma(R_{\theta_j} (x_0 - A_{j,k_j}^{-1} \lambda_j))$.
Hence, for $j \gg 0$,
\begin{align}
\label{eq_supp_gtj}
g^2_j \equiv 0
\mbox{ on } B_{\gamma/2}(0).
\end{align}
Since $(1-\eta)f \in \mathcal{S}'$, there exist $N$ such that for all $h \in \mathcal{S}$,
\begin{align*}
\abs{\ip{(1-\eta)f}{h}} \lesssim \sum_{\abs{\alpha} \leq N}
\norm{(1+\abs{\cdot})^N \partial^\alpha h}_\infty.
\end{align*}
Using this estimate we have that for all $j \geq 1$,
\begin{align*}
&\abs{\ip{g^2_j}{h}}=
\abs{\ip{T_{-D_j^{-1} \lambda_j} \dil_{R_{-\theta_j}}(1-\eta)f}{h}}
\\
&\qquad =
\abs{\ip{(1-\eta)f}{\dil_{R_{\theta_j}} T_{D_j^{-1} \lambda_j} h}}
\lesssim \sum_{\abs{\alpha} \leq N}
\norm{(1+\abs{\cdot})^N \partial^\alpha (\dil_{R_{\theta_j}} T_{D_j^{-1} \lambda_j} h)}_\infty.
\\
&\qquad \lesssim \sum_{\abs{\alpha} \leq N}
\norm{(1+\abs{\cdot})^N \partial^\alpha T_{D_j^{-1} \lambda_j} h}_\infty
\\
&\qquad \lesssim \sum_{\abs{\alpha} \leq N}
\norm{(1+\abs{\cdot})^N \partial^\alpha h}_\infty,
\end{align*}
where the last two bounds follow from the fact that $\abs{R_{\theta_j}x}=\abs{x}$ and
that, according to \eqref{eq_best_lambdaj},
\begin{align*}
\norm{D_j^{-1} \lambda_j}=\norm{D_j^{-1} (\lambda_j-A_{j, k_j} x_0) + R_{\theta_j} x_0}
\leq 2^{-j} L_\Lambda + \norm{x_0} \lesssim 1.
\end{align*}

Taking into account the vanishing property in \eqref{eq_supp_gtj},
the bound on $\abs{\ip{g^2_j}{h}}$ can be improved for $j \gg 0$ to
\begin{align}
\label{eq_improved_gtj}
\abs{\ip{g^2_j}{h}}
\lesssim \sum_{\abs{\alpha} \leq N}
\norm{(1+\abs{\cdot})^N \partial^\alpha h}_{L^\infty(\Rtst \setminus B_{\gamma'}(0))},
\end{align}
where $\gamma':=\gamma/4$.

We finally use this to bound $\ip{g^2_j}{\varphi_j}$. Since $\varphi \in \mathcal{S}$ we have that for all
$L>2N+2$ there is a constant $C_L>0$ such that for all multi-indices with $\abs{\alpha} \leq N$
\begin{align*}
\abs{\partial^\alpha \varphi (x)} \leq C_L \abs{x}^{-L}.
\end{align*}
Hence, when $\abs{x} \geq \gamma'$,
\begin{align*}
&(1+\abs{x})^N \abs{\partial^\alpha \varphi_j (x)}
\leq
C_{\gamma'} \abs{x}^N \abs{\partial^\alpha \varphi_j (x)}
\\
&\qquad \lesssim 
\abs{x}^N 8^{j/2} 4^{jN} 
\sum_{\beta \leq \alpha}
\abs{\partial^\beta \varphi (4^jx_1, 2^jx_2)}
\\
&\qquad \lesssim C_L \abs{x}^N 4^{j(N+1)} \abs{(4^jx_1, 2^jx_2)}^{-L}
\\
&\qquad \leq C_L 4^{j(N+1)} 2^{-jL} \abs{x}^{N-L} \leq C_L {\gamma'}^{(N-L)} 2^{j(2N+2-L)}.
\end{align*}
Plugging this into \eqref{eq_improved_gtj} shows
that $\abs{\ip{g^2_j}{\varphi_j}} \lesssim 2^{-(2N+2-L)j}$, for all $L > 2N+2$. This clearly implies the desired
estimate.

\step{Step 2}. We show that $\sum_{j \geq 1} \abs{\ip{g^1_j}{\varphi_j}}^2 4^{2js} < +\infty$.

Since $g^1_j =  T_{-D_j^{-1} \lambda_j} \dil_{R_{-\theta_j}} (\eta f)$, we have
\begin{align*}
\abs{\widehat{g^1_j(\xi)}}=\abs{\widehat{\eta f}(R_{-\theta_j} \xi)}, \quad (\xi \in \Rtst).
\end{align*}
So if we let $\alpha_j := 2\pi-\theta_0-\theta_j$ and 
$\Phi(\xi) := \widehat{\eta f}(R_{\theta_0} \xi)$,  we see  that
\begin{align}
\label{eq_abs_g1}
\abs{\widehat{g^1_j(\xi)}}=\Phi(R_{\alpha_j} \xi), \quad (\xi \in \Rtst).
\end{align}
Note that from \eqref{eq_best_kj} we get
\begin{align}
0 \leq \alpha_j \leq 2\pi 2^{-j},\qquad (j\geq 1).
\end{align}
Since $f$ satisfies \eqref{eq_big_bound_micro}, it follows that $\Phi$ satisfies,
\begin{align}
\label{eq_big_bound_Phi}
\int_{\Rtst} \Phi(\xi)^2 (1+\abs{\xi})^{-2M} \,d\xi <+\infty.
\end{align}
In addition, by \eqref{eq_micro_cone}, we have
\begin{align}
\label{eq_phi_int}
\int_{\cone_\varepsilon} \Phi(\xi)^2 \abs{\xi}^{2s} \,d\xi <+\infty,
\end{align}
where $\varepsilon>0$ and the cone $\cone_\varepsilon$ is defined in \eqref{eq_def_cone}.
Let $k \in \Nst$ be such that $2^{-k} < \varepsilon$. Since $\alpha_j \rightarrow 0$, there exist
$j_0 \geq 1$ such that for all $j \geq j_0$,
\begin{align}
\label{eq_inclusion_cones}
R_{-\alpha_j} (\cone_{2^{-k}}) \subseteq \cone_\varepsilon.
\end{align}
Consider the sets $\sett{W_{r,t}: r,t \in \Zst}$ from \eqref{eq_sets_Wrt}.
Since these sets partition $\Rtst$ up to sets of measure zero, we can decompose
$\widehat{\varphi}$ as
\begin{align*}
&\widehat{\varphi} = \sum_{r,t \in \Zst} \widehat{\varphi^{r,t}},
\mbox{ a.e.}
\\
&\widehat{\varphi^{r,t}} := \widehat{\varphi} \cdot 1_{W_{r,t}}.
\end{align*}
Let us define $K_{r,t}\geq 0$ by
\begin{align}
\label{eq_deff_Kst}
K_{r,t}^2 :=
\int_{\Rtst} \abs{\widehat{\varphi^{r,t}}(\xi)}^2 \, d\xi
=
\int_{W_{r,t}} \abs{\widehat{\varphi}(\xi)}^2 \, d\xi.
\end{align}

Our main assumption \eqref{eq_assum_decay_window_2_plus} on the window $\varphi$
implies that
\begin{align*}
K_{r,t}^2 \lesssim 4^{r(2\moment+1)}e^{-2\delta {16}^{r-1}} 2^t e^{-2\delta 4^{t-1}}.
\end{align*}
Hence if $\newa,\newb \in \Rst$
are such that $\newa < (2\moment+1)$ and $\newb > -1$, then
\begin{align}
\label{eq_estimate_Krt}
\sum_{r,t \in \Zst} 4^{-\newa r} 2^{\newb t} K_{r,t}^2 < +\infty.
\end{align}

We also use the notation $\widehat{\varphi^{r,t}_j}(\xi) := 
\dil_{D_{-j}}\widehat{\varphi^{r,t}}(\xi)
=8^{- j/2} \widehat{\varphi^{r,t}}(4^{-j}\xi_1,2^{-j}\xi_2)$,
so for each $j \geq 1$
\begin{align*}
\widehat{\varphi_j} = \sum_{r,t \in \Zst} \widehat{\varphi^{r,t}_j}.
\end{align*}
Fix $r,t \in \Zst$ and let us bound $\ip{g^1_j}{\varphi^{r,t}_j}$.

Since $\widehat{\varphi^{r,t}}$ is supported on $W_{r,t}$, its dilation $\widehat{\varphi^{r,t}_j}$ is
supported on $D_j W_{r,t}$. Recall the numbers $k,j_0$ from \eqref{eq_inclusion_cones}.
The set $D_j W_{r,t}$ is contained in the cone $V_{2^{-k}}$, if $j \geq k+t-2r+2$.
Let us denote $j_1=j_1(r,t):=\max\{j_0, k+t-2r+2\}$. For $j \geq j_1(r,t)$ we bound,
\begin{align*}
&\abs{\ip{g^1_j}{\varphi^{r,t}_j}}
\leq 8^{- j/2} \int_{D_j W_{r,t}} \abs{\widehat{g^1_j}(\xi)}
\abs{\widehat{\varphi^{r,t}}(4^{-j}\xi_1,2^{-j}\xi_2)} \, d\xi
\\
&\qquad
\leq
8^{- j/2} 4^{-js} \int_{D_j W_{r,t}}
\abs{\widehat{g^1_j}(\xi)} \abs{\xi_1}^{s} \abs{4^{-j} \xi_1}^{-s}
\abs{\widehat{\varphi^{r,t}}(4^{-j}\xi_1,2^{-j}\xi_2)} \, d\xi
\\
&\qquad
\leq
4^{-js}
\left( \int_{D_j W_{r,t}} \abs{\widehat{g^1_j}(\xi)}^2 \abs{\xi_1}^{2s} \, d\xi \right)^{1/2}
8^{- j/2}
\left(
\int_{D_j W_{r,t}} \abs{4^{-j} \xi_1}^{-2s}
\abs{\widehat{\varphi^{r,t}}(4^{-j}\xi_1,2^{-j}\xi_2)}^2 \, d\xi \right)^{1/2}
\\
&\qquad
\leq
4^{-js}
\left( \int_{D_j W_{r,t}} \abs{\widehat{g^1_j}(\xi)}^2 \abs{\xi_1}^{2s} \, d\xi \right)^{1/2}
\left(
\int_{W_{r,t}} \abs{\xi_1}^{-2s} \abs{\widehat{\varphi^{r,t}}(\xi)}^2 \, d\xi
\right)^{1/2}
\\
&\qquad
\lesssim
4^{-js} 4^{-rs}
\left( \int_{D_j W_{r,t}} \abs{\widehat{g^1_j}(\xi)}^2 \abs{\xi_1}^{2s} \, d\xi \right)^{1/2}
\left(
\int_{W_{r,t}} \abs{\widehat{\varphi^{r,t}}(\xi)}^2 \, d\xi
\right)^{1/2}.
\end{align*}
Hence, since $\abs{\xi_1}\leq\abs{\xi}$,
\begin{align*}
\abs{\ip{g^1_j}{\varphi^{r,t}_j}} \leq 4^{-js} 4^{-rs}
K_{r,t}
\left( \int_{D_j W_{r,t}} \abs{\widehat{g^1_j}(\xi)}^2 \abs{\xi}^{2s} \, d\xi \right)^{1/2},
\qquad (j \geq j_1(r,t)).
\end{align*}
Using \eqref{eq_abs_g1} we get
\begin{align*}
\abs{\ip{g^1_j}{\varphi^{r,t}_j}} \leq 4^{-js} 4^{-rs} K_{r,t}
\left( \int_{R_{-\alpha_j} (D_j W_{r,t})} \Phi(\xi)^2 \abs{\xi}^{2s} \, d\xi \right)^{1/2},
\qquad (j \geq j_1(r,t)).
\end{align*}
For $j \geq j_1(r,t)$, we have that $D_j W_{r,t} \subseteq V_{2^{-k}}$ so
\eqref{eq_inclusion_cones} implies that
$R_{-\alpha_j} (D_j W_{r,t}) \subseteq \cone_{\varepsilon}$. Combining
this with \eqref{eq_phi_int} and the bound on the number of overlaps of the sets
$\sett{R_{-\alpha_j} (D_j W_{r,t}): j \geq 1}$
granted by Lemma \ref{lemma_rot_overlaps}, we obtain
\begin{align}
\label{eq_bound_bigj}
&\sum_{j \geq j_1(r,t)} \abs{\ip{g^1_j}{\varphi^{r,t}_j}}^2 4^{2sj}
\lesssim  4^{-2rs} K_{r,t}^2.
\end{align}
Since by assumption $s<\moment+1/2$, we use \eqref{eq_estimate_Krt} 
with $\newa = 2 s$ and $\newb = 0$ and obtain
\begin{align}
\label{eq_sum_rt_1}
\sum_{r,t \in \Zst} \sum_{j \geq j_1(r,t)}
\abs{\ip{g^1_j}{\varphi^{r,t}_j}}^2 4^{2sj} < + \infty.
\end{align}
For $1 \leq j < j_1(r,t)$ we use the bound
\begin{align*}
(1+\abs{\xi}) \leq 1+ 4^j\abs{D_{-j} \xi} \leq 4^j(1+\abs{D_{-j} \xi}),
\end{align*}
and carry out a similar estimate.
\begin{align*}
&\abs{\ip{g^1_j}{\varphi^{r,t}_j}}
\leq 8^{- j/2} \int_{D_j W_{r,t}} \abs{\widehat{g^1_j}(\xi)}
\abs{\widehat{\varphi^{r,t}}(D_{-j}\xi)} \, d\xi
\\
&\qquad
\leq
8^{- j/2} 4^{jM} \int_{D_j W_{r,t}}
\abs{\widehat{g^1_j}(\xi)} (1+\abs{\xi})^{-M} (1+\abs{D_{-j} \xi})^{M}
\abs{\widehat{\varphi^{r,t}}(D_{-j}\xi)} \, d\xi
\\
&\qquad
\leq
4^{jM}
\left( \int_{D_j W_{r,t}} \abs{\widehat{g^1_j}(\xi)}^2 (1+\abs{\xi})^{-2M} \, d\xi \right)^{1/2}
8^{- j/2}
\left(
\int_{D_j W_{r,t}} (1+\abs{D_{-j}\xi})^{2M} 
\abs{\widehat{\varphi^{r,t}}(D_{-j}\xi)}^2 \, d\xi \right)^{1/2}
\\
&\qquad
\leq
4^{jM}
\left( \int_{D_j W_{r,t}} \abs{\widehat{g^1_j}(\xi)}^2 (1+\abs{\xi})^{-2M} \, d\xi \right)^{1/2}
\left(
\int_{W_{r,t}} (1+\abs{\xi})^{2M} \abs{\widehat{\varphi^{r,t}}(\xi)}^2 \, d\xi
\right)^{1/2}
\\
&\qquad
\lesssim
4^{jM} (1+4^r+2^t)^M
\left( \int_{D_j W_{r,t}} \abs{\widehat{g^1_j}(\xi)}^2 (1+\abs{\xi})^{-2M} \, d\xi \right)^{1/2}
\left(
\int_{W_{r,t}} \abs{\widehat{\varphi^{r,t}}(\xi)}^2 \, d\xi
\right)^{1/2}.
\end{align*}
Therefore, for $1 \leq j < j_1(r,t)$,
\begin{align*}
\abs{\ip{g^1_j}{\varphi^{r,t}_j}} \leq 4^{jM} (1+4^r+2^t)^M
K_{r,t}
\left( \int_{R_{-\alpha_j} (D_j W_{r,t})} \Phi(\xi)^2 (1+\abs{\xi})^{-2M} \, d\xi \right)^{1/2}.
\end{align*}
Hence, for $1 \leq j < j_1(r,t)$,
\begin{align*}
\abs{\ip{g^1_j}{\varphi^{r,t}_j}} 4^{js} \leq 4^{j_1(r,t) (M+s)} (1+4^r+2^t)^M
K_{r,t}
\left( \int_{R_{-\alpha_j} (D_j W_{r,t})} \Phi(\xi)^2 (1+\abs{\xi})^{-2M} \, d\xi \right)^{1/2}.
\end{align*}
Using again the bound on the number of overlaps of $\sett{R_{-\alpha_j} (D_j W_{r,t}): j \geq 1}$
of Lemma \ref{lemma_rot_overlaps}, this time combined
with \eqref{eq_big_bound_Phi}, we get
\begin{align*}
&\sum_{j=j_0}^{j_1(r,t)-1}
\abs{\ip{g^1_j}{\varphi^{r,t}_j}}^2 4^{2sj} 
\lesssim 4^{2 j_1(r,t) (M+s)} (1+4^r+2^t)^{2M} K^2_{r,t}.
\end{align*}
To bound this quantity we may assume that $j_1(r,t)=k+t-2r+2>j_0$, since otherwise $j_1(r,t)=j_0$ and the sum is empty.
Hence,
\begin{align*}
&\sum_{j=j_0}^{j_1(r,t)-1}
\abs{\ip{g^1_j}{\varphi^{r,t}_j}}^2 4^{2sj} 
\lesssim 4^{2 (t-2r) (M+s)} (1+4^r+2^t)^{2M} K^2_{r,t}
\\
&\qquad
\lesssim 2^{4(M+s)t} 4^{-4(M+s)r} (1+4^{2Mr}+2^{2Mt}) K^2_{r,t}
\\
&\qquad
= \left(4^{-(4M+4s)r} 2^{(4M+4s)t}
+
4^{-(2M+4s)r} 2^{(4M+4s)t}
+
4^{-(4M+4s)r} 2^{(6M+4s)t}
\right)K^2_{r,t}.
\end{align*}
To bound the sum over $r,t \in \Zst$ of the last expression, we
use \eqref{eq_estimate_Krt} with
$\newa = 4M+4s, \newb=4M+4s$; $\newa=2M+4s, \newb=4M+4s$;
$\newa = 4M+4s, \newb=6M+4s$. Note that in each case
$\newa \leq 4M + 4s < 2 \moment + 1$ by assumption and $\newb \geq 0 > -1$.
Hence the use of \eqref{eq_estimate_Krt} is justified and we conclude that
\begin{align}
\label{eq_sum_rt_2}
\sum_{r,t \in \Zst} \sum_{j=j_0}^{j_1(r,t)-1}
\abs{\ip{g^1_j}{\varphi^{r,t}_j}}^2 4^{2sj} < + \infty.
\end{align}
Finally, since $\varphi_j = \sum_{r,t} \varphi_j^{r,t}$,
\begin{align*}
&\sum_{j \geq j_0} \abs{\ip{g^1_j}{\varphi_j}}^2 4^{2js}
\leq
\sum_{r,t \in \Zst}
\sum_{j \geq j_1(r,t)} \abs{\ip{g^1_j}{\varphi^{r,t}_j}}^2 4^{2js}
+
\sum_{r,t \in \Zst}
\sum_{j=j_0}^{j_1(r,t)-1} \abs{\ip{g^1_j}{\varphi^{r,t}_j}}^2 4^{2js}
< + \infty.
\end{align*}
\end{proof}

\subsection{Estimates for the approximate reconstruction}
Theorem \ref{th_decay_wave} shows that the coefficients of the wavepacket expansion
in \eqref{eq_frame_expansion_2} decay fast away from the
wavefront set of $f$. We now show that the same is true for the coefficients in
the approximate wavepacket expansion in \eqref{eq_app_exp}. To avoid immaterial
technicalities, let us assume that $\winzero, \varphi \in \mathcal{S}(\Rtst)$.
\begin{theorem}
\label{th_decay_wave_approx}
In the setting of Theorem \ref{th_dual}, let $(x_0,\theta_0) \in \Rtst \times [0,2\pi)$
and let $\sett{(k_j, \lambda_j): j\geq 1}$ be the corresponding grid parameters
given by \eqref{eq_best_kj} and \eqref{eq_best_lambdaj}.
If $f \in \LtRt$ belongs to $H^s$ microlocally at $(x_0,\theta_0)$, then
\begin{align*}
\sum_{j \geq 1} \abs{\ip{f}{\widetilde{\varphi}_{j,k_j,\lambda_j}}}^2 4^{2js} < +\infty.
\end{align*}
\end{theorem}
As mentioned in Remark \ref{rem_mult}, the approximate coefficients can be rewritten as
\begin{align*}
\ip{f}{\widetilde{\varphi}_{j,k,\lambda}}=\ip{M_{1/\widetilde{m}} f}{\psi_{j,k,\lambda}},
\end{align*}
where $\psi_{j,k,\lambda}=\dilAjk T_\lambda \psi$, $\widetilde{m}$ is given by \eqref{eq_approx_m}
and $\widehat M_{1/\widetilde{m}}(f) := (1/\widetilde{m}) \widehat{f}$. Since $\psi$ satisfies the same conditions
that
$\varphi$, it follows from Theorem \ref{th_decay_wave} that the numbers $\ip{f}{\widetilde{\varphi}_{j,k,\lambda}}$
decay away from the $H_s$-wavefront set of $M_{1/\widetilde{m}}(f)$. Hence, it suffices to show that $f$ and
$M_{1/\widetilde{m}}(f)$ share the same $H_s$-wavefront set. We do so in the next lemma.

\begin{lemma}
Under the hypothesis of Theorem \ref{th_dual}, let $s \geq 0$. Then a function $f \in \LtRt$ 
belongs to the microlocal Sobolev space $H^s(x_0,\theta_0)$ if and only if $M_{1/\widetilde{m}}(f)$ does.
Consequently,
$f$ and $M_{1/\widetilde{m}}(f)$ have the same $H_s$-wavefront set.
\end{lemma}
\begin{proof}
Let $g:=M_{1/\widetilde{m}}(f)$. We want to show that $g$ and $M_{\widetilde{m}} g$ have the same 
$H^s$-wavefront set.
The Fourier multiplier $M_{\widetilde{m}}$ is a pseudo-differential operator with Weyl symbol $\sigma(x,\xi)
:= \widetilde{m}(\xi)$. Under the hypothesis of Theorem \ref{th_dual}, by \eqref{eq_approx_covering},
$0< \widetilde{A} \leq \widetilde{m}(\xi) \leq \widetilde{B} < +\infty$.
This implies that $M_{\widetilde{m}}$ is elliptic of order 0. We will
show that $\sigma$ belongs to H\"ormander's symbol class $S^0_{\frac{1}{2},0}$, i.e.,
\begin{align*}
\abs{\partial^\alpha_{\xi} \partial^\beta_x \sigma(x,\xi)}
=\abs{\partial^\alpha_{\xi} \widetilde{m}(\xi)} \leq C_\alpha (1+\abs{\xi})^{-\frac{\abs{\alpha}}{2}}
\mbox{, for every multi-index }\alpha.
\end{align*}
Once this is established, it will follow that $M_{\widetilde{m}}$ preserves the $H_s$-wavefront set
(see \cite[Chapter 18]{ho85} or \cite[Chapter 2]{folland89}). Let us write
\begin{align*}
\widetilde{m}(\xi) := \abs{\widehat{\winzero}(\xi)}^2 + \widetilde{m_1}(\xi),
\end{align*}
where
\begin{align*}
&\widetilde{m_1}(\xi) := \sum_{j \geq 1} \sum_{k=0}^{2^j-1} \Phi(B_{j,k} \xi),
\\
&\Phi(\xi):= \widehat{\varphi}(\xi)\overline{\widehat{\psi}(\xi)}.
\end{align*}
Since $\abs{\widehat{\winzero}(\xi)}^2$ is a Schwartz function and
thus in $S^{0}_{\tfrac{1}{2},0}$ we focus on $\widetilde{m_1}$.
Note that, by the construction of $\psi$, $\Phi$ is a compactly supported
smooth function that vanishes in some neighborhood of the origin $(-\varepsilon,\varepsilon)^2$.

Let $\alpha$ be a multi-index, $N=\abs{\alpha}$ and
\begin{align*}
\Phi_N(\xi) := \sum_{\abs{\beta} \leq N} \abs{\partial^\beta_{\xi} \Phi(\xi)}.
\end{align*}
Therefore
\begin{align}
\label{eq_sum_derivatives}
\abs{\partial^\alpha_{\xi} \widetilde{m_1}(\xi)} \leq \sum_{j \geq 1} \sum_{k=0}^{2^j-1} 2^{-jN}
\Phi_N(B_{j,k} \xi).
\end{align}
For each $j,k$, $\abs{B_{j,k} \xi} \geq 4^{-j}\abs{\xi}$ and consequently
$\abs{\xi}^{N/2} \leq 2^{jN} \abs{B_{j,k}\xi}^{N/2}$. Hence, setting
$\Phi^*_N(\xi):=\abs{\xi}^{N/2}\Phi_N(\xi)$ we have that
\begin{align*}
&\abs{\xi}^{N/2} \abs{\partial^\alpha_{\xi} \widetilde{m_1}(\xi)} \leq
\sum_{j \geq 1} \sum_{k=0}^{2^j-1} \Phi^*_N(B_{j,k} \xi) \leq \norm{\Phi^*_N}_*.
\end{align*}
Since the support of $\Phi^*_N$ is compact and does not contain $0$, Proposition
\ref{prop_scale_correlation} implies that $\norm{\Phi^*_N}_* <+\infty$.
Similarly,
\begin{align*}
\abs{\partial^\alpha_{\xi} \widetilde{m_1}(\xi)}
\leq
\sum_{j \geq 1} \sum_{k=0}^{2^j-1} 2^{-jN} \Phi_N(B_{j,k} \xi)
\leq
\sum_{j \geq 1} \sum_{k=0}^{2^j-1} \Phi_N(B_{j,k} \xi) \leq \norm{\Phi_N}_* < +\infty.
\end{align*}
The last two estimates imply that
$\abs{\partial^\alpha_{\xi} \widetilde{m_1}(\xi)} \lesssim
(1+\abs{\xi})^{-\frac{N}{2}}=(1+\abs{\xi})^{-\frac{\abs{\alpha}}{2}}$, as desired.
\end{proof}

\end{document}